\documentclass[a4paper,11pt]{amsart}
\usepackage[letterpaper, margin=1in]{geometry}
\usepackage{latexsym}
\usepackage{amssymb}
\usepackage{amsmath}
\usepackage{amsfonts}
\usepackage{color, soul}
\usepackage{tikz}

\newtheorem{thm}{Theorem}[section]

\newtheorem{lem}[thm]{Lemma}

\newtheorem{cor}[thm]{Corollary}

\theoremstyle{definition}

\newcommand{\N}{\mathbb N}
\newcommand{\Z}{\mathbb Z}

\newcommand{\F}{\mathbb F}

\def\al{\alpha}

\def\s{\sigma}
\def\d{\delta}

\newcounter{cs}
\stepcounter{cs}
\newcommand{\casos}{\begin{itemize}}
\newcommand{\fcasos}{\end{itemize}\setcounter{cs}{1}}

\newfont{\tit}{cmr12 scaled \magstep3}

\begin{document}

\title[]{Characterizations and properties of principal $(f, \sigma, \delta)$-codes over  rings}

\author{Mhammed Boulagouaz}
\address[M. Boulagouaz]{Department of Mathematics, Faculty of Sciences and Technologies, University of Sidi Mohamed Ben Abdellah, B.P. 2202, Fes, Morocco} \email{boulag@rocketmail.com}

\author{Abdulaziz Deajim}
\address[A. Deajim]{Department of Mathematics, King Khalid University,
P.O. Box 9004, Abha, Saudi Arabia} \email{deajim@kku.edu.sa, deajim@gmail.com}

\date{\today}

\begin {abstract}
Let $A$ be a ring with identity, $\sigma$ a ring endomorphism of $A$ that maps the identity to itself,
$\delta$ a $\sigma$-derivation of $A$, and consider the skew-polynomial ring $A[X;\sigma,\delta]$.
When $A$ is a finite field, a Galois ring, or a general ring, some fairly recent literature used $A[X;\sigma,\delta]$
to construct new interesting codes (e.g. skew-cyclic and skew-constacyclic codes) that generalize their classical counterparts over finite fields
(e.g. cyclic and constacyclic linear codes). This paper presents results concerning {\it principal} $(f, \sigma, \delta)$-codes over a ring $A$, where $f\in A[X;\sigma,\delta]$ is monic. We provide recursive formulas that compute the entries of both a generating matrix and a control matrix of such a code $\mathcal{C}$. When further $A$ is a finite commutative ring and $\sigma$ is a ring automorphism of $A$, we also give recursive formulas for the entries of a parity-check matrix of $\mathcal{C}$. Also in this case, with $\delta=0$, we present a characterization of principal $\sigma$-codes whose dual codes are also principal $\sigma$-codes, and we deduce a characterization of self-dual principal $\sigma$-codes. Some corollaries concerning principal $\sigma$-constacyclic codes are also given, and a good number of highlighting examples is provided.
\end {abstract}
\maketitle

\section{{\bf Introduction}}\label{intro}

\subsection{State of the Art}
Let $A$ be a ring with identity, $\sigma$ a ring endomorphism of $A$ that maps the identity to itself, and $\delta$ a $\sigma$-derivation of $A$ (i.e. $\delta(a+b)=\delta(a)+\delta(b)$ and $\delta(ab)=\sigma(a)\delta(b)+\delta(a)b$ for all $a,b \in A$). Denote by $A_{\sigma,\delta}$ the skew-polynomial ring
$$A[X; \sigma, \delta]=\left\{ \sum_{i=0}^{n-1} a_i X^i \,|\, n\in \N, \, a_i\in A\right\}.$$
Recall that $A_{\sigma,\delta}$ has the same additive-group structure as that of the usual ring of polynomials $A[X]$, but has multiplication twisted based on the rule $Xa=\sigma(a)X+\delta(a)$ for $a\in A$ and extended associatively and distributively to all elements of $A_{\sigma, \delta}$. This is obviously a non-commutative ring unless $\d=0$, $\s$ is the identity, and $A$ is commutative (in which case $A_{\s, \d}$ is nothing but $A[X]$). In case $\d=0$, we use the notation $A_\s$ instead of $A_{\s,0}$.

For a finite field $\F$ and a ring automorphism $\s$ of $\F$, Boucher, Geiselmann, and Ulmer (in \cite{BGU}) used $\F_\s$ to introduce the notion of a skew-cyclic code $\mathcal{C}$ over $\F$ of length $n$ as a code satisfying $(\s(a_{n-1}), \s(a_0), \s(a_1), \dots, \s(a_{n-2})) \in \mathcal{C}$ for any $(a_0, a_1, \dots, a_{n-2}, a_{n-1}) \in \mathcal{C}$.
This is obviously a generalization of the classical notion of cyclic codes over finite fields (when $\s$ is the identity). It is also shown therein that the class of skew-cyclic codes over finite fields gives a supply of codes with good coding and decoding properties. When a monic $f\in \F_\s$ generates a two-sided ideal in $\F_\s$, then $\F_\s/(f)$ is a (non-commutative) principal left-ideal ring. In particular, when the order of $\s$ divides $n$, then $(X^n-1)$ is a two sided ideal in $\F_\s$ (see \cite{BGU}). When, further, $g\in \F_\s$ is a right divisor of $X^n-1$, the authors of \cite{BGU} studied the skew-cyclic code generated by $g$ and associated with the principal left ideal $(g)/(X^n-1)$ of $\F_\s/(X^n-1)$. The structure of such an ideal puts some restrictions on the code (for instance, $(X^n-1)$ must be a two-sided ideal, which is ensured by some arithmetical condition on $n$).

To further generalize the notion of skew-cyclic codes, Boucher and Ulmer in \cite{BU1} introduced codes defined as modules over $\F_\s$. Among other things, this new construction has the advantage of removing some of the constraints on the lengths of skew-cyclic codes alluded to above. Boucher, Sol\'{e}, and Ulmer in \cite{BSU} relaxed the requirement on the field of coefficients by considering skew-polynomial rings over Galois rings enabling further generalizations and improvements (see also \cite{DNS}). Boulagouaz and Leroy in \cite{BL} took this generalization further by letting the ring of coefficients be any ring $A$ with $\s$ a ring endomorphism of $A$ and $\d$ a $\s$-derivation of $A$. A nice recent generalization in a different direction may be found in \cite{BGS}.

\subsection{Contributions and Map of the Article}

\begin{itemize}
\item For a ring $A$ (not necessarily finite nor commutative), an endomorphism $\s$ of $A$, and a $\s$-derivation $\d$ of $A$, the following is done:
\begin{enumerate}
\item In Section 2, we revisit the main definitions of \cite{BL} and particularly make precise the \linebreak notions of principal $(f, \s, \d)$-codes, $\s$-codes, $(f,\s,\d)$-constacyclic codes, and \linebreak $\s$-constacyclic codes over $A$.
\item Section \ref{generating matrix} aims mainly at computationally improving \cite[Theorem 1]{BL} by giving a \linebreak generating matrix of a principal $(f, \s, \d)$-code (resp. a principal $\s$-code) over $A$ \linebreak using recursive formulas introduced by means of a list of lemmas; see Theorem \ref{main 0} (resp. Corollary \ref{A_sigma}).
\item In Section 5, we present precise and more practical recursive formulas which yield, in Theorem \ref{main 1}, the entries of a control matrix of a principal $(f,\s,\d)$-code $\mathcal{C}$ over $A$ whose generating polynomial is both a right and left divisor of $f$. This gives Theorem \ref{main 1} a practicality advantage over \cite[Corollary 1]{BL}. Furthermore, for a principal $\s$-code (resp. a principal $\s$-constacyclic code) over $A$, the control matrix of given in Theorem \ref{main 1} takes a better shape; see Corollary \ref{control delta zero} (resp. Corollary \ref{control consta}).
\item In the beginning of Section 6, we introduce the notion of a parity-check matrix of an $A$-free $(f, \s, \d)$-code.
\end{enumerate}
\item For a finite commutative ring $A$ and an automorphism $\s$ of $A$, the following is done:
\begin{enumerate}
\item In Section 4, we characterize principal $\s$-codes over $A$ whose dual codes are also \linebreak principal $\s$-codes, strengthening and extending \cite[Theorem 1]{BU2}; see Theorem \ref{dual 1} and the paragraph that precedes it. Consequently, we give in Corollary \ref{main 2} a generating matrix of the dual of a principal $\s$-constacyclic code over $A$. We further introduce, in Corollary \ref{SDC}, a characterization of self-dual principal $\s$-codes over $A$ in such a way that generalizes and strengthens \cite[Corollary 4]{BU1}.
\item In Section 6, we construct a parity-check matrix of a principal $(f, \s, \d)$-code $\mathcal{C}$ over $A$ and show how to extract such a matrix from a control matrix of $\mathcal{C}$; see Theorem \ref{main 3}. Furthermore, for a principal $\s$-code (resp. a principal $\s$-constacyclic code) over $A$, the parity-check matrix given in Theorem \ref{main 3} takes a better shape; see Corollary \ref{parity sigma code} (resp. Corollary \ref{parity-consta}). On the other hand, with the crucial assumption that $\s$ is an automorphism of $A$, we show in Corollary \ref{parity-consta2} that the parity-check matrix given in Corollary \ref{parity-consta} can be obtained without the assumption that the principal $\s$-constacyclic code is generated by some monic $g(X)\in A_\s$ that is also a left divisor of $X^n-a$.
\end{enumerate}
\item Throughout the article, a good number of highlighting examples is given. Some results from this article were used to construct novel matrix-product codes arising from $(\s, \d)$-codes (see \cite{BD}).
\end{itemize}

\section{{\bf Preliminaries}}

Let $A$ be a ring with identity, $\s$ a ring endomorphism of $A$ that maps the identity to itself, and $\d$ a $\s$-derivation of $A$. Fix a monic skew-polynomial $f(X)=\sum_{i=0}^n a_i X^i \in A_{\sigma,\delta}$ of degree $n$. In order to define the notion of $(f,\s, \d)$-code, we begin by using $f$ to endow $A^n$ with a structure of a left $A_{\s, \d}$-module. Let $C_f$ be the usual companion matrix of $f$; that is,
$$C_f =  \left( \begin{array}{ccccccc} 0 & 1 & 0 & 0 & \dots & 0 & 0\\ 0 & 0 & 1 & 0 & \dots & 0 & 0\\ \vdots & \vdots & \vdots & \vdots & \vdots & \vdots & \vdots\\ 0 & 0 & 0 & 0 & \dots & 0 & 1\\ -a_0 & -a_1 &  & \dots&  & -a_{n-2} & -a_{n-1} \end{array} \right).$$
The map $T_f:A^n \to A^n$ defined by
$$T_f(x_0, \dots, x_{n-1}) = (\sigma(x_0), \dots, \sigma(x_{n-1}))C_f + (\delta(x_0), \dots, \delta(x_{n-1}))$$
is a {\it $(\sigma, \delta)$-pseudo-linear transformation (associated to $f$)}; that is, considering $A^n$ as a left $A$-module, we have $T_f(ax) = \sigma(a)T_f(x)+\delta(a)x$ for all $a\in A$ and $x\in A^n$. It can also be easily checked that $T_f$ is a group endomorphism of $A^n$ (see \cite{BL} for more details and examples on this notion). For a skew-polynomial $P(X)=\sum_{i=0}^{n-1} b_iX^i \in A_{\sigma,\delta}$, the map $P(T_f)=\sum_{i=0}^{n-1} b_iT_f^i$ is obviously a group endomorphism of $A^n$ as well. Now, the map $(P(X), (c_0, \dots, c_{n-1})) \mapsto P(T_f)(c_0, \dots, c_{n-1})$ defines a left action of $A_{\sigma,\delta}$ on $A^n$, which in turn endows $A^n$ with a left $A_{\sigma, \delta}$-module structure as desired.

For an $(f,\s,\d)$-code $\mathcal{C}$, the set $\{y\in A^n\, |\, <x,y>\,=\,0 \;\mbox{for all}\; x\in \mathcal{C}\}$ of elements of $A^n$ orthogonal to $\mathcal{C}$ with respect to the Euclidean inner product on $A^n$ is called the dual of $\mathcal{C}$ and is denoted by $\mathcal{C}^\perp$. It can be checked that $\mathcal{C}^\perp$ is a left $A$-submodule of $A^n$.

Let $(f)_l$ denote the principal left ideal of $A_{\sigma, \delta}$ generated by $f$. With $A^n$ and $A_{\s, \d}/(f)_l$ as left $A_{\s, \d}$-modules, the map $\phi_f:A^n \to A_{\sigma,\delta}/(f)_l$ defined by $(d_0, \dots, d_{n-1}) \mapsto \sum_{i=0}^{n-1} d_i X^i + (f)_l$ is a left $A_{\sigma,\delta}$-module isomorphism. The coset $\sum_{i=0}^{n-1} d_i X^i + (f)_l$ is called the {\it polynomial representation} of $(d_0, \dots, d_{n-1})$ in $A_{\sigma,\delta}/(f)_l$. On the other hand, we know that for each $t(X) \in A_{\sigma,\delta}$, there exists a unique $p(X)=\sum_{i=0}^{n-1} d_iX^i \in A_{\sigma, \delta}$ of degree at most $n-1$ such that $t(X)+(f)_l = p(X)+(f)_l$. The $n$-tuple $(d_0, \dots, d_{n-1})\in A^n$ is called the {\it coordinates} of $\, t(X)+(f)_l$ (with respect to the basis $\mathcal{B}=\{1+(f)_l, X + (f)_l, \dots, X^{n-1} + (f)_l\}$). Note that $(d_0, \dots, d_{n-1})=\phi_f^{-1}(t(X)+(f)_l)$.

An $(f,\s,\d)$-code of length $n$ over $A$ is a linear code $C\subseteq A^n$ such that
$(x_0, \dots, x_{n-1}) \in C$ implies that $T_f(x_0, \dots, x_{n-1})\in C$ (see \cite{BL}). With the above notation, an {\it $(f, \sigma, \delta)$-code of length $n$ over $A$} is a subset $\mathcal{C}$ of $A^n$ consisting of the coordinates of a left $A_{\sigma, \delta}$-submodule $\mathcal{M}$ of $A_{\sigma, \delta}/(f)_l$ with respect to $\mathcal{B}$, i.e. $\mathcal{C}=\phi_f^{-1}(\mathcal{M})$ for some left $A_{\s, \d}$-submodule $\mathcal{M}$ of $A_{\s,\d}/(f)_l$. Equivalently, $\mathcal{C} \subseteq A^n$ is an $(f, \sigma, \delta)$-code if and only if the set $\phi_f(\mathcal{C})$ of polynomial representations of elements of $\mathcal{C}$ is a left $A_{\sigma,\delta}$-submodule of $A_{\sigma,\delta}/(f)_l$.
So, there is a one-to-one correspondence between $(f, \sigma, \delta)$-codes over $A$ and left $A_{\sigma,\delta}$-submodules of $A_{\sigma,\delta}/(f)_l$. If $\delta=0$, an $(f, \sigma, \delta)$-code may be called an $(f, \sigma)$-code, or just a $\s$-code if $f$ is irrelevant to the context. A linear code $\mathcal{C}\subseteq A^n$ is called a $(\s, \d)$-code of length $n$ if there exists a monic skew-polynomial $f\in A_{\s, \d}$ of degree $n$ such that $\mathcal{C}$ is an $(f, \s, \d)$-code.

{A ring over which every two bases of any finitely generated free (right) module have the same (finite) number of elements is said to have (right)
{\it Invariant Basis Number} (IBN for short). This common number is defined to be {\it the rank} of such a module. Examples of such rings include nonzero commutative rings, nonzero finite rings, division rings, and local rings. For more on IBN rings, see \cite[Chapter 1]{Lam}. From now on, whenever we mention the finite rank of a free module, we implicitly assume without mention that the underlying ring is IBN.}

As $\mathcal{M}$ is a left $A_{\s, \d}$-submodule of $A_{\s, \d}/(f)_l$, $\mathcal{C}$ is a left $A_{\s, \d}$-submodule of $A^n$. Then, note {\it a priori} that $\mathcal{M}$ and $\mathcal{C}$ are left $A$-modules, and $\mathcal{M}$ is free over $A$ of rank $r$ if and only if $\mathcal{C}$ is free over $A$ of rank $r$.

An $(f, \sigma, \delta)$-code $\mathcal{C}=\phi_f^{-1}(\mathcal{M})$ over $A$ is said to be {\it principal} if the left $A_{\sigma, \delta}$-submodule $\mathcal{M}$ of $A_{\sigma, \delta}/(f)_l$ is generated by a skew-polynomial $g\in A_{\sigma, \delta}$ that is a right divisor of $f$ in $A_{\sigma, \delta}$ (so, $\displaystyle{\mathcal{M}=(g)_l/(f)_l}$). We may also say that $\mathcal{M}$ is {\it principal} in this case. Note that since $f$ is monic, the leading coefficient of $g$ must be a unit $u$ in $A$. But then $u^{-1}g \in A_{\sigma, \delta}$ is also a monic right divisor of $f$ that generates the same left ideal $(g)_l$. So it causes no harm to always assume that $g$ is monic to begin with. A linear code $\mathcal{C}\subseteq A^n$ is called a principal $(\s, \d)$-code of length $n$ if there exists a monic skew-polynomial $f\in A_{\s, \d}$ of degree $n$ such that $\mathcal{C}$ is a principal $(f, \s, \d)$-code. Rephrased according to our terminology, \cite[Theorem 1]{BL}) shows that a principal $(f, \s, \d)$-code generated by a skew polynomial $g$ is free over $A$ of rank equal to $\mbox{deg}(f)-\mbox{deg}(g)$. It should be noted, however, that not all $(f,\sigma, \delta)$-codes are principal since not all left $A_{\sigma, \delta}$-submodules of $A_{\sigma, \delta}/(f)_l$ are principal.

In the special case when $f(X)=X^n -a$ for some unit $a\in A$ and $\mathcal{M}$ is a left $A_{\s, \d}$-submodule (resp. a principal left $A_{\s, \d}$-submodule) of $A_{\sigma, \delta}/(X^n-a)_l$, the $(X^n-a, \s, \d)$-code $\mathcal{C}=\phi_{X^n-a}^{-1}(\mathcal{M})$ is called an $(X^n-a, \s, \d)$-{\it constacyclic} (resp. a principal $(X^n-a, \s, \d)$-{\it constacyclic}) code. We deal in this paper with such a code only when $\d=0$ and thus call it an $(X^n-a, \s)$-constacyclic (resp. a principal $(X^n-a, \s)$-constacyclic) code. A linear code $\mathcal{C}\subseteq A^n$ is called a $\s$-constacyclic code (resp. a principal $\s$-constacyclic code) of length $n$ if there exists a unit $a\in A$ such that $\mathcal{C}$ is an $(X^n-a, \s)$-code (resp. a principal $(X^n-a, \s)$-code). A principal $\s$-constacyclic code generated by a right divisor $g(X)\in A_\s$ of $X^n-a$, for some unit $a\in A$, is denoted by $(g(X))_{n,\s}^a$.

\section{{\bf Generating Matrix of a Principal $(f, \sigma, \delta)$-Code over a Ring}}\label{generating matrix}

We work in this section assuming that $A$ is a ring with identity, $\sigma$ is a ring endomorphism of $A$ that maps the identity to itself, and $\d$ a $\s$-derivation of $A$. For an $A$-free $(f, \s, \d)$-code $\mathcal{C}$ of rank $n-r$, define a {\it generating matrix} of $\mathcal{C}$ to be a matrix $G \in M_{n-r, n}(A)$ whose rows form an $A$-basis of $\mathcal{C}$ (see \cite{Lint} for the classical definition of a generating matrix of a linear code over a field). In set notation, we thus have $$\mathcal{C}= \{x\,G\,|\, x\in A^{n-r}\}.$$

Let $f(X)=\sum_{i=0}^n a_i X^i\in A_{\s, \d}$ be monic and $\mathcal{C}$ a principal $(f,\s,\d)$-code generated by a monic $g(X)=\sum_{i=1}^r g_i X^i \in A_{\s, \d}$ of degree $r$. Then, by \cite[Theorem 1]{BL}, $\mathcal{C}$ is free over $A$ of rank $n-r$. Using $g(X)$ and the map $T_f$ introduced in Section 1, Boulagouaz and Leroy in \cite{BL} gave a way of computing $G$ as in Lemma \ref{B-L} below. The main aim of this section is to introduce, in Theorem \ref{main 0}, more practical recursive formulas that compute the entries of $G$ using $g(X)$, $\s$, and $\d$. Corollary \ref{A_sigma} deals with the case when $\d=0$.

\begin{lem}\label{B-L} \cite[Theorem 1]{BL}
With the assumptions as above, the principal $(f, \sigma, \delta)$-code $\mathcal{C}$ has a generating matrix $G\in M_{n-r, n}(A)$ whose rows are given by
$$T_f^k(g_0, \dots, g_r, 0, \dots, 0)$$ for $0 \leq k \leq n-r-1$.
\end{lem}

The following results aim at giving the set-up for producing formulas that compute $$T_f^k(g_0, \dots, g_r, 0, \dots, 0),$$ and hence $G$, much more easily.

To simplify notation, for $k\geq 0$ and $(x_0, \dots, x_{n-1}) \in A^n$, we set
\begin{align*}
(x_0^{(k)}, \dots, x_{n-1}^{(k)})&= T_f^k(x_0, \dots, x_{n-1}).
\end{align*}

\begin{lem} \label{lemma L}
For $(x_0,\dots,x_{n-1})\in A^n$ and $k\in \N$, we have:
\begin{itemize}
\item[(a)] $x_0^{(k)}=\delta(x_0^{(k-1)})-\sigma(x_{n-1}^{(k-1)})a_0$.

\item[(b)] $x_i^{(k)}=\delta(x_{i}^{(k-1)})+
\sigma(x_{i-1}^{(k-1)})-\sigma(x_{n-1}^{(k-1)})a_i$, for $1\leq i\leq n-1$.
\end{itemize}
\end{lem}

\begin{proof}
By definition,
\begin{align*}
T_f(x_0,x_1, \dots, x_{n-1})&= (\sigma(x_0),\sigma(x_1), \dots, \sigma(x_{n-1}))\,C_f + (\delta(x_0), \delta(x_1),\dots, \delta(x_{n-1}))\\
&=(\delta(x_0)- \sigma(x_{n-1})a_0, \delta(x_1)+\sigma(x_0)-\sigma(x_{n-1})a_1, \dots,\\ &\qquad \qquad \delta(x_{n-1})+\sigma(x_{n-2})-\sigma(x_{n-1})a_{n-1}).
\end{align*}
As $(x_0^{(k)}, x_1^{(k)}, \dots, x_{n-1}^{(k)})=T_f^k(x_0, \dots, x_{n-1})=T_f(x_0^{(k-1)}, x_1^{(k-1)}, \dots, x_{n-1}^{(k-1)})$, we have
\begin{align*}(x_0^{(k)}, x_1^{(k)}, \dots, x_{n-1}^{(k)})= ( \delta(x_0^{(k-1)})- \sigma(x_{n-1}^{(k-1)})a_0, \delta(x_1^{(k-1)})+\sigma(x_0^{(k-1)})-\sigma(x_{n-1}^{(k-1)})a_1,\\ \dots, \delta(x_{n-1}^{(k-1)})+\sigma(x_{n-2}^{(k-1)})-\sigma(x_{n-1}^{(k-1)})a_{n-1}).\end{align*}
\end{proof}

\begin{cor}\label{lemma LL}
For $(x_0, \dots, x_{n-1}) \in A^n$ with $x_{n-1}=0$, we have:
\begin{itemize}
\item[(a)] $x_0^{(1)} =\delta(x_0)$.
\item[(b)] $x_i^{(1)}=\delta(x_i)+\sigma(x_{i-1})$ for $1\leq i \leq n-2$.
\item[(c)] $x_{n-1}^{(1)}=\sigma(x_{n-2})$.
\end{itemize}
\end{cor}

\begin{proof}
This follows directly from Lemma \ref{lemma L} and properties of $\sigma$ and $\delta$.
\end{proof}

\begin{cor}\label{enhance}
For $(x_0, \dots, x_{n-1}) \in A^n$ with $x_{s+1}=\dots =x_{n-1}=0$ for some $0\leq s \leq n-2$, we have:
\begin{itemize}
\item[(a)] $x_{s+i}^{(i)}=\sigma^i (x_s)$ for $1 \leq i < n-s-1$.
\item[(b)] $x_{s+j}^{(i)} =0$ for $1 \leq i < j \leq n-s-1$.
\end{itemize}
\end{cor}

\begin{proof}
We proceed by (finite) induction on $i$. For $i=1$, it follows from Corollary \ref{lemma LL} that
$$x_{s+1}^{(1)} = \delta(x_{s+1})+\sigma(x_s) = \delta(0)+\sigma(x_s)=\sigma(x_s).$$
For $1=i<j \leq n-s-1$, we have $s+1\leq s+j-1\leq n-2$ and (by Corollary \ref{lemma LL})
$$x_{s+j}^{(1)} =\delta(x_{s+j})+\sigma(x_{s+j-1})=\delta(0)+\sigma(0)=0.$$

Assume now, for $1<i<n-s-1$, that $x_{s+i-1}^{(i-1)} = \sigma^{i-1}(x_s)$ and, for $i-1<t\leq n-s-1$, that $x_{s+t}^{(i-1)}=0$. Then it follows from Lemma \ref{lemma L} that $$x_{s+i}^{(i)}=\delta(x_{s+i}^{(i-1)})+\sigma(x_{s+i-1}^{(i-1)})-\sigma(x_{n-1}^{(i-1)})a_{s+i}=\delta(0)+
\sigma(\sigma^{i-1}(x_s))-\sigma(0)a_{s+i}=\sigma^i(x_s).$$

We also have (by Lemma \ref{lemma L}), for $1<i<j\leq n-s-1$, that
$$x_{s+j}^{(i)} =\delta(x_{s+j}^{(i-1)}) +\sigma(x_{s+j-1}^{(i-1)})-\sigma(x_{n-1}^{(i-1)})a_{s+j}=\delta(0) + \sigma(x_{s+j-1}^{(i-1)}) -\sigma(0)a_{s+j}=\sigma(x_{s+j-1}^{(i-1)}).$$
As $i-1< j-1$, $x_{s+j-1}^{(i-1)}=0$ by assumption. Thus, $x_{s+j}^{(i)}=0$ as claimed.
\end{proof}

\begin{cor}\label{enhance 2}
Let $(x_0, \dots, x_{n-1}) \in A^n$ and $\d=0$.
\begin{itemize}
\item[(a)] If $x_{n-1}=0$, then $x_0^{(1)}=0$ and $x_i^{(1)}=\s(x_{i-1})$ for $1 \leq i \leq n-1$.

\item[(b)] If $x_{s+1}=\dots =x_{n-1}=0$ for some $0 \leq s \leq n-2$, then for any $1\leq k \leq n-s-1$,
\begin{itemize}
\item[(i)] $x_i^{(k)}=0$ for $0 \leq i \leq k-1$, and
\item[(ii)] $x_i^{(k)} =\s(x_{i-1}^{(k-1)})$ for $0 \leq k-1 < i \leq n-1$.
\end{itemize}
\end{itemize}
\end{cor}

\begin{proof}\hfill
\begin{itemize}
\item[(a)] A direct application of Corollary \ref{lemma LL} yields the claim.
\item[(b)] We proceed by (finite) induction on $k$. Let $k=1$. If $0\leq i \leq k-1$, then $i=0$. So $x_0^{(k)}=x_0^{(1)}=0$ by part (1) above. From part (1) again, for $0= k-1 < i\leq n-1$, $x_i^{(k)}=x_i^{(1)}=\s(x_{i-1})=\s(x_{i-1}^{(0)})=\s(x_{i-1}^{(k-1)})$ as desired. Assume now that the result holds for all $1\leq k < n-s-1$. Set $y_i=x_i^{(k)}$ for each $0\leq i \leq n-1$, and note that $y_i^{(t)}=(x_i^{(k)})^{(t)}=x_i^{(k+t)}$ for all $t\geq 1$. By the inductive assumption, we see that $$y_{n-1}=x_{n-1}^{(k)}=\s(x_{n-2}^{(k-1)})=\s^2(x_{n-3}^{(k-2)})=\cdots = \s^k(x_{n-1-k}^{(0)})=\s^k(x_{n-1-k}).$$ As $k< n-s-1$, $n-1-k>s$. So, $x_{n-1-k}=0$ and, thus, $y_{n-1}=0$. It now follows from part (1) applied to $(y_0, \dots, y_{n-1})$ that $x_0^{(k+1)}=y_0^{(1)}=0$ and, for $1\leq i \leq n-1$, $x_i^{(k+1)}=y_i^{(1)}=\s(y_{i-1})=\s(x_{i-1}^{(k)})$. Note, in particular, that for $1\leq i\leq k+1$, $0\leq i-1 \leq k$. So, $x_{i-1}^{(k)}=0$ by the inductive assumption and, therefore, $x_i^{(k+1)}=\s(0)=0$ in this case.
\end{itemize}
\end{proof}

\begin{cor}\label{enhance 3}
Let $(x_0, \dots, x_{n-1}) \in A^n$, $\d=0$, and $a_1=\cdots =a_{n-1}=0$. Then,
\begin{itemize}
\item[(a)] For $k\in \N $, we have:

\indent \indent $1$. $x_0^{(k)} = - \s(x_{n-1}^{(k-1)})a_0$.

\indent \indent $2$. $x_t^{(k)}=\s(x_{t-1}^{(k-1)})$ for $1 \leq t \leq n-1$.

\item[(b)] If, further, $x_0=x_1=\dots x_s=0$ for some $0 \leq s \leq n-2$, then we have:
\begin{itemize}
\item[(i)]
 $1$. $x_0^{(1)}=- \s(x_{n-1})a_0$.

\indent \indent $2$. $x_t^{(1)}=0$ for $1\leq t \leq s+1$.

\indent \indent $3$. $ x_t^{(1)}=\s(x_{t-1})$ for $s+2\leq t\leq n-1$.

\item[(ii)] For $2 \leq i \leq n-s-1$, we have:

\indent \indent $1$. $x_0^{(i)}=- \s(x_{n-1}^{(i-1)})a_0$.

\indent \indent $2$. $x_t^{(i)} =\s(x_{t-1}^{(i-1)})$ for $1\leq t \leq i-1$.

\indent \indent $3$. If $s\geq 1$, then $x_t^{(i)}=0$ for $i\leq t \leq i+s-1$

\indent \indent $4$. $x_t^{(i)}=\s(x_{t-1}^{(i-1)})$ for $i+s \leq t \leq n-1$.
\end{itemize}
\end{itemize}
\end{cor}

\begin{proof}\hfill
\begin{itemize}
\item[(a)] This follows directly from Lemma \ref{lemma L} with $\d=0$ and $a_1=\cdots =a_{n-1}=0$.
\item[(b)] By part (a), we have:
\begin{itemize}
\item[(i)]
 1. $x_0^{(1)}=- \s(x_{n-1}^{(0)})a_0=- \s(x_{n-1})a_0$.

\indent \indent 2. For $1\leq t \leq s+1$, $x_t^{(1)}=\s(x_{t-1}^{(0)})=\s(x_{t-1})=\s(0)=0$.

\indent \indent 3. For $s+2 \leq t \leq n-1$, $x_t^{(1)}=\s(x_{t-1}^{(0)})=\s(x_{t-1})$.

\item[(ii)] Items 1,2, and 4 are immediate from part (a). As for item 3, assume that $s\geq 1$. We use (finite) induction on $i$. For $i=2$ and $2\leq t \leq s+1$, we have $1\leq t-1\leq s$ and it thus follows from part (a) and part (b-i-2) that $x_t^{(2)}=\s(x_{t-1}^{(1)})=\s(0)=0$. Suppose now that $x_t^{(i)}=0$ for $2\leq i \leq n-s-2$ and $i\leq t \leq i+s-1$. Then for $i+1\leq t \leq i+s$, we have $i\leq t-1 \leq i+s-1$. So, it follows from part (a) and the inductive step that $x_t^{(i+1)}=\s(x_{t-1}^{(i)})=\s(0)=0$.
\end{itemize}
\end{itemize}
\end{proof}

Now comes the main result of this section, which gives precise recursive formulas for the entries of a generating matrix of $\mathcal{C}$ enhancing \cite[Theorem 1]{BL}.

\begin{thm} \label{main 0}
Keep the assumptions mentioned at the beginning of this section. Then,
a generating matrix $G\in M_{n-r,n}(A)$ of $\mathcal{C}$ is

$$ \left( \begin{array}{ccccccc} g_0 & \dots  & g_r & 0& 0 &  \dots & 0\\ g_0^{(1)} & \dots & g_r^{(1)} & \sigma(g_r) & 0 &\dots & 0\\ g_0^{(2)} & \dots & g_r^{(2)} & g_{r+1}^{(2)} & \sigma^2(g_r) &\dots & 0\\  \vdots & \vdots & \vdots & \vdots  &\vdots& \vdots & \vdots \\ g_0^{(n-r-1)} & \dots & g_r^{(n-r-1)} & g_{r+1}^{(n-r-1)} & g_{r+2}^{(n-r-1)} &\dots & \sigma^{n-r-1}(g_r) \end{array} \right) ,$$ where
\begin{itemize}
\item[(1)] $g_i=0$ for $r+1 \leq i \leq n-1$,
\item[(2)] $g_0^{(i)} =\d(g_0^{(i-1)})$ for $1\leq i \leq n-r-1$, and
\item[(3)] $g_j^{(i)}=\d(g_j^{(i-1)})+\s( g_{j-1}^{(i-1)})$ for $1\leq i \leq n-r-1$ and $1\leq  j\leq n-2$.
\end{itemize}
\end{thm}

\begin{proof}
Using Lemma \ref{B-L} and applying Corollary \ref{enhance} with $s=r$, $(x_1, \dots, x_n)=(g_0, \dots, g_{n-1})$, and $g_{r+1}=\cdots = g_n=0$ yield the claim of the theorem.
\end{proof}

If $\mathcal{C}$ of Theorem \ref{main 0} is a principal $\s$-code (i.e. $\d=0$), then a generating matrix of $\mathcal{C}$ takes a more beautiful form as the following result indicates, the proof of which is just a direct application of Theorem \ref{main 0} in this special case.

\begin{cor}\label{A_sigma}
Keep all the assumptions of Theorem \ref{main 1} with $\d=0$. Then, a generating matrix $G\in M_{n-r,n}(A)$ of $\mathcal{C}$ is
$$\left( \begin{array}{cccccccc} g_0  & \cdots  & g_r & 0& 0 &  \cdots & 0\\ 0 & \sigma(g_0)  & \cdots & \sigma(g_r) & 0 &\dots & 0\\ \vdots & \ddots  & &   &\ddots&  & \vdots \\ 0 & 0 & \cdots &0 &   \sigma^{n-r-1}(g_0) & \cdots & \sigma^{n-r-1}(g_r) \end{array} \right).$$
\end{cor}

\noindent{\bf Example 1.} Let $R$ be a ring with identity and $A$ the ring $\left\{\left(\begin{array}{cc} a & b \\ 0 & a \\\end{array}\right)\,|\, a,b \in R \right\}$. Letting \linebreak $\displaystyle{\s: \left(\begin{array}{cc}a & b \\ 0 & a \\\end{array}\right) \mapsto \left(\begin{array}{cc}a & 0 \\ 0 & a \\\end{array}\right)}$ and $\delta:\left(\begin{array}{cc}a & b \\ 0 & a \\\end{array}\right) \mapsto \left(\begin{array}{cc}0 & b \\ 0 & 0\\\end{array}\right)$, it can be checked that $\s$ is a ring endomorphism of $A$ that maps the identity to itself and $\d$ is a $\s$-derivation of $A$. Let $\mathcal{C}$ a principal $(\s, \d)$-code of length 4 generated by $g(X)= X-\al \in A_{\s, \d}$ with $\al=\left(\begin{array}{cc} 1 & 1\\ 0 & 1 \end{array}\right)$. Noting that $g_0=-\al$, $g_1=1$, $g_2=g_3=0$, we get from Theorem \ref{main 0} that a generating matrix of $\mathcal{C}$ is
\begin{align*}
G&=\left(
      \begin{array}{ccccc}
        g_0 & g_1& 0 & 0\\
        g_0^{(1)} & g_1^{(1)} & \s(g_1) & 0\\
        g_0^{(2)} & g_1^{(2)} & g_2^{(2)} & \s^2(g_1)  \\
      \end{array}
    \right)
    =\left(
      \begin{array}{cccc}
        g_0 & 1 & 0 & 0\\
        \delta(g_0) & \delta(1) + \sigma(g_0) & 1 & 0\\
          ( \delta(g_0))^{(1)} & (\delta(1) + \sigma(g_0))^{(1)} & (1)^{(1)} & (0)^{(1)}  \\
      \end{array}
    \right)\\
  &=\left(
      \begin{array}{cccc}
        g_0 & 1 & 0 & 0\\
        \delta(g_0) & \delta(g_1) + \sigma(g_0) & 1 & 0\\
          \delta^2(g_0) & \delta^2(1)+(\delta \sigma+  \sigma\delta)(g_0)&
          (\delta \sigma +\sigma\delta)(1) + \sigma^2(g_0) & 1  \\
      \end{array}
    \right)\\
    &=\left(
      \begin{array}{cccc}
        -\al & 1 & 0 & 0\\
        \d(-\al) & \s(-\al) & 1 & 0\\
          \delta^2(-\al) & 0 & \s^2(-\al) & 1  \\
      \end{array}
    \right)
    =\left(
      \begin{array}{cccc}
        -\alpha & 1 & 0 & 0\\
       1-\alpha & -1 &1 & 0  \\
        1-\alpha & 0&-1 &1  \\
      \end{array}
    \right).
\end{align*}
On the other hand, if $\d=0$, then it follows from Corollary \ref{A_sigma} that a generating matrix of $\mathcal{C}$ is
\begin{align*}
G&=\left(
      \begin{array}{ccccc}
        g_0 & g_1& 0 & 0\\
        0 & \s(g_0) & \s(g_1) & 0\\
        0 & 0 & \s^2(g_0) & \s^2(g_1)  \\
      \end{array}
    \right)
    =\left(
      \begin{array}{cccc}
        -\al & 1 & 0 & 0\\
        0 & -1 & 1 & 0\\
        0 & 0 & -1 & 1  \\
      \end{array}
    \right).
    \end{align*}\\

\noindent{\bf Example 2.}
Let $A=\F_3\times \F_3$, $\sigma(x,y)=(y,x)$, and
$f(X) =X^6+1\in A_\s$. Denoting $(a,a)\in A$ by $a$, we can see that
$f(X)=(X^2+1)(X^4+2X^2+1)=(X^4+2X^2+1)(X^2+1)$. The $\s$-code generated by $g(X)=X^4+2X^2+1$ is a principal $\s$-constacyclic (or negacyclic if one wishes), which is a self-orthogonal $[6,4,2]$ code over $A$ with generating matrix
$\begin{pmatrix}
  1& 0 & 2 & 0 & 1 & 0\\
  0 &  1& 0 & 2 & 0 & 1
\end{pmatrix}$. Using the obvious Gray map on Magma (\cite{M}), this code yields a ternary $[12,4,3]$ code whose dual is a $[12,8,2]$ code, which is quasi-optimal (see \cite{Gr}).

\section{{\bf The Dual of a Principal $\s$-Code over a Finite Commutative Ring}}\label{dual code}
We assume in this section that $A$ is a finite commutative ring with identity, $\s$ is an automorphism of $A$, and $U(A)$ is the multiplicative group of units of $A$. We give in Theorem \ref{dual 1} a characterization of principal $\s$-codes over $A$ whose duals are also principal $\s$-codes, strengthening and extending \cite[Theorem 1]{BU2}. Furthermore, Corollary \ref{main 2} utilizes Theorem \ref{dual 1} to give a generating matrix of the dual of a principal $\s$-constacyclic code. Finally, Corollary \ref{SDC} characterizes self-dual principal $\s$-codes over $A$ in such a way that generalizes and strengthens \cite[Corollary 4]{BU1}.

For a skew-polynomial $h(X)=\sum_{i=0}^s h_i X^i \in A_\s$, define the following skew-polynomials:
$$\s^n(h(X))=\sum_{i=0}^s \s^n(h_i)X^i \;(\mbox{for} \;n \in \mathbb{N}) \quad \mbox{and} \quad h^*(X)=\sum_{i=0}^s \s^i(h_{s-i})X^i.$$
Consider the ring of Laurent skew-polynomials:
$$A[X, X^{-1}; \s]=\left\{\sum_{i=-m}^n a_i X^i \,|\, m,n\in \mathbb{N}\cup\{0\},\, a_i\in A\right\} ,$$
where addition is given by the usual rule and multiplication is given by the rule $$(a_iX^i)(b_jX^j)=a_i\s^i(b_j)X^{i+j} \quad (\mbox{for}\; i,j \in \mathbb{Z})$$ and then extending associatively and distributively to all elements of $A[X, X^{-1}; \s]$. It is \linebreak obvious that $A_\s$ is a subring of $A[X, X^{-1}; \s]$. It is worth noting that $X^{-1}a=\s^{-1}(a)X^{-1}$ and $\displaystyle{aX^{-1}=X^{-1}\s(a)}$ for all $a\in A$.

The following result and its proof are similar, in part, to their counterparts over finite fields appearing in the literature (see for instance \cite[Lemma 1]{BU2}).

\begin{lem}\label{lemma 1}
Let $\psi:A[X, X^{-1}; \s] \to A[X, X^{-1}; \s]$ be the map defined by $$\sum_{i=-m}^n a_i X^i \mapsto \sum_{i=-m}^n X^{-i}a_i,$$ and let $h(X)=\sum_{i=0}^s h_i X^i \in A_\s$ be of degree $s$. Then the following hold:
\begin{itemize}
\item[(i)] $\psi$ is a ring anti-automorphism,
\item[(ii)] $h^*(X)=X^s\psi (h(X))$,
\item[(iii)] for any $n\in \mathbb{N}$, $X^n h(X)=\s^n(h(X))X^n$, and
\item[(iv)] if $h_s$ is not a zero divisor in $A$, then $h(X)$ is not a zero divisor in $A_\s$.
\end{itemize}
\end{lem}

\begin{proof}$\\$
\begin{itemize}
\item[(i)] It is straightforward to show that $\psi$ is bijective and additive. Consider two Laurent skew-polynomials $S(X)=\sum_{i=-m_1}^{n_1} s_i X^i$ and $T(X)=\sum_{j=-m_2}^{n_2} t_jX^j$. Letting $k=\mbox{max}\{m_1,m_2\}$, we may add zero terms if necessary to set $S(X)=\sum_{i=-k}^{n_1} s_i X^i$ and $T(X)=\sum_{j=-k}^{n_2} t_jX^j$.
Then,
\begin{align*}
\psi(S(X)T(X))&= \psi((\sum_{i=-k}^{n_1} s_i X^i) (\sum_{j=-k}^{n_2} t_jX^j))= \psi(\sum_{i=-k}^{n_1}\;\;\sum_{j=-k}^{n_2} s_i X^i t_j X^j)\\
&= \psi(\sum_{j=-k}^{n_2} \;\;\sum_{i=-k}^{n_1} s_i \s^i(t_j) X^{i+j} = \sum_{j=-k}^{n_2}\;\; \sum_{i=-k}^{n_1} X^{-(i+j)}s_i\s^i(t_j).
\end{align*}
On the other hand,
\begin{align*}
\psi(T(X))\psi(S(X))&= (\sum_{j=-k}^{n_2} X^{-j} t_j) (\sum_{i=-k}^{n_1} X^{-i}s_i) = \sum_{j=-k}^{n_2}\;\;\sum_{i=-k}^{n_1} X^{-j}t_j X^{-i}s_i\\
&=\sum_{j=-k}^{n_2}\;\;\sum_{i=-k}^{n_1} X^{-(i+j)}\s^i(t_j)s_i = \sum_{j=-k}^{n_2}\;\;\sum_{i=-k}^{n_1} X^{-(i+j)}s_i\s^i(t_j).
\end{align*}
Thus, $\psi(S(X)T(X))=\psi(T(X))\psi(S(X))$.
\item[(ii)] We see that \begin{align*}
X^s\psi(h(X))&=X^s(\sum_{i=0}^s X^{-i}h_i)=\sum_{i=0}^s X^{s-i}h_i\\
&=\sum_{i=0}^s \s^{s-i}(h_i)X^{s-i}=\sum_{j=0}^s \s^j(h_{s-j})X^j\\
&=h^*(X).
\end{align*}
\item[(iii)] $X^n h(X)= \sum_{i=0}^s X^n h_i X^i = \sum_{i=0}^s \s^n(h_i) X^{i+n}= \s^n(h(X))X^n.$
\item[(iv)] Let $r(X)=\sum_{j=0}^t r_jX^j \in A_\s$ be such that $r_t\neq 0$ and $h(X)r(X)=0$ (resp. $r(X)h(X)=0$). Then, $\s^s(r_t)h_s=0$ (resp. $r_t \s^t(h_s)=0$). Note that since $h_s$ is not a zero divisor in $A$ and $\s^t$ is an automorphism of $A$, $\s^t(h_s)$ is not a zero divisor in $A$ either. It then follows that $\s^s(r_t)=0$ (resp. $r_t=0$). Since $\s^s$ is an automorphism of $A$, it follows in both cases that $r_t=0$, a contradiction. Thus, $r(X)=0$.
\end{itemize}
\end{proof}

Special cases of the following two results (in the context of finite fields) appear in \cite{BU2}.

\begin{lem}\label{lemma 2}
Let $g(X)=\sum_{i=0}^{n-k} g_i X^i \in A_\s$ be of degree $n-k$, $g_{n-k}\in U(A)$, $\displaystyle{h(X)=\sum_{i=0}^k h_i X^i \in A_\s}$ of degree $k$, and $b\in U(A)$. Then, $X^n-b=g(X)h(X)$ if and only if $X^n-a=\s^n(h(X))g(X)$ for $a=\s^k(b)\s^{k-n}(g_{n-k})\s^k(g_{n-k}^{-1})$.

\end{lem}

\begin{proof}
Either of the claimed equivalent statements imply that $h_{k} \in U(A)$. We first prove the lemma for the case when $g(X)$ is monic. Assume that $X^n-b = g(X)h(X)$. Then $h(X)$ is monic too. It follows from Lemma \ref{lemma 1} (iii) that, $$\s^n(h(X))g(X)h(X)=\s^n(h(X))X^n-\s^n(h(X))b=X^n h(X) -\s^n(h(X))b.$$ So, $[X^n-\s^n(h(X))g(X)]h(X)=\s^n(h(X))b$. Since $\mbox{deg}(h(X))=\mbox{deg}(\s^n(h(X))b)$ and $h(X)$ is monic, $\mbox{deg}(X^n-\s^n(h(X))g(X))=0$ regardless of the characteristic of $A$. So, $X^n-\s^n(h(X))g(X)=a$ for some nonzero $a\in A$, and $ah(X)-\s^n(h(X))b=0$. Since $h(X)$ and $\s^n(h(X))$ are monic, the leading coefficient of $ah(X)-\s^n(h(X))b$ is $a-\s^k(b)$. Thus, $a=\s^k(b)$ and $X^n-\s^k(b)=\s^n(h(X))g(X)$ as claimed. Conversely, suppose that $X^n-\s^k(b)=\s^n(h(X))g(X)$. Applying the above argument for $\s^n(h(X))$ and $\s^k(b)$ instead of $g(X)$ and $b$, respectively, yields $$X^n-\s^{n-k}(\s^k(b))=\s^n(g(X))\s^n(h(X)).$$ So, $\s^n(X^n-b)=\s^n(g(X)h(X))$ and, thus, $X^n-b=g(X)h(X)$ as claimed.

We now drop the assumption that $g(X)$ is monic. Assume that $X^n-b = g(X)h(X)$ and let $G(X)=g_{n-k}^{-1}g(X)$. Then $G(X)\in A_\s$ is monic, and
\begin{align*}
G(X)h(X)&=g_{n-k}^{-1}X^n-g_{n-k}^{-1}b\\
& =X^n \s^{-n}(g_{n-k}^{-1})-bg_{n-k}^{-1}\\
& =[X^n-b \s^{-n}(g_{n-k})g_{n-k}^{-1}]\s^n(g_{n-k}^{-1}).
\end{align*}
Letting $H(X)=h(X)\s^{-n}(g_{n-k})\in A_\s$, we then have $G(X)H(X)=X^n-b \s^{-n}(g_{n-k})g_{n-k}^{-1}$. Since $G(X)$ is monic and $b \s^{-n}(g_{n-k})g_{n-k}^{-1}\in U(A)$, it follows from the argument in the first paragraph of this proof that
\begin{align*}
X^n-\s^k(b)\s^{k-n}(g_{n-k})\s^k(g_{n-k}^{-1})&=X^n -\s^k(b \s^{-n}(g_{n-k})g_{n-k}^{-1})\\
&=\s^n(H(X))G(X)\\
&=\s^n(h(X))g_{n-k}G(X)\\
&=\s^n(h(X))g(X)
\end{align*}
as claimed.

Conversely, suppose that $X^n-a=\s^n(h(X))g(X)$ with $a=\s^k(b)\s^{k-n}(g_{n-k})\s^k(g_{n-k}^{-1})$. Note that $a\in U(A)$ since $\s$ is an automorphism of $A$ and $g_{n-k}\in U(A)$. Let $G(X)=g_{n-k}^{-1}g(X)$. Then, $\displaystyle{G(X)\in A_\s}$ is monic and $X^n -a=\s^n(h(X))g_{n-k}G(X)$. As $h(X)g_{n-k}\in A_\s$ and $\s^k$ and $\s^n$ are automorphisms of $A$ (and also additive automorphisms when extended to $A_\s$), let $\displaystyle{c\in U(A)}$ and $\displaystyle{H(X)\in A_\s}$ be such that $a=\s^k(c)$ and $\s^n(h(x))g_{n-k}=\s^n(H(X))$. So, $\displaystyle{X^n-\s^k(c)=H(X)G(X)}$. It now follows from the argument in the first paragraph of this proof that $X^n-c=G(X)H(X)$; that is $$X^n-\s^{-k}(a)=G(X)h(X)\s^{-n}(g_{n-k})=g_{n-k}^{-1}g(X)h(X)\s^{-n}(g_{n-k}).$$ So,
\begin{align*}
g_{n-k}[X^n-\s^{-k}(a)]&=g(X)h(X)\s^{-n}(g_{n-k}), \\
X^n \s^{-n}(g_{n-k})-g_{n-k}\s^{-k}(a)&=g(X)h(X)\s^{-n}(g_{n-k}),\\
[X^n-g_{n-k}\s^{-k}(a)\s^{-n}(g_{n-k}^{-1})]\s^{-n}(g_{n-k})&=g(X)h(X)\s^{-n}(g_{n-k}),\\
X^n-g_{n-k}\s^{-k}(a)\s^{-n}(g_{n-k}^{-1})&=g(X)h(X)\s^{-n}(g_{n-k})\s^{-n}(g_{n-k}^{-1}),\\
X^n-g_{n-k}b\s^{-n}(g_{n-k})g_{n-k}^{-1}\s^{-n}(g_{n-k}^{-1})&=g(X)h(X).
\end{align*}
Hence, $X^n-b=g(X)h(X)$ as claimed.
\end{proof}

\noindent{\bf Remark:} If we do not want to be so specific about the nature of $a, b,$ and $h(X)$ as they appear above, we could rephrase Lemma \ref{lemma 2} as follows:
\begin{quote}
 \it{A skew-polynomial $g(X)\in A_\s$, whose leading coefficient is a unit in $A$, is a left divisor of $X^n-b\in A_\s$ for some $b\in U(A)$ if and only if $g(X)$ is a right divisor of $X^n-a\in A_\s$ for some $a\in U(A)$.}
 \end{quote}

\noindent{\bf Example 3.} Let $\s$ be an automorphism of $A$, and $\al\in U(A)$ with $\s(\al)=\al$. For $g(X)=X-\al$ and $h(X)= X^3+\al X^2 +\al^2 X +\al^3$, we have $X^4-\al^4= g(X)h(X)$ in $A_\s$. On the other hand, $$\s^4(h(X))g(X)=h(X)g(X)=X^4 - \s^3(\al^4)\s^{-1}(1)\s^3(1^{-1})=X^4-\al^4$$ as asserted by Lemma \ref{lemma 2}.

\begin{lem}\label{h*} Let $h(X)=\sum_{i=0}^k h_iX^i \in A_\s$ be of degree $k$ with $h_0\in U(A)$. If $h(X)$ is a right divisor of $X^n-b$ for some $b\in U(A)$ , then $h^*(X)$ is a left divisor of $X^n-\s^{k-n}(b^{-1})$ and a right divisor of $X^n-b^{-1}\s^{-k}(h_0)\s^{n-k}(h_0^{-1})$.
\end{lem}

\begin{proof}
Suppose that $h(X)$ is a right divisor of $X^n-b$ for some $b\in U(A)$. So $l(X)h(X)=X^n-b$ for some $l(X)\in A_\s$ with $\mbox{deg}(l(X))=n-k$. We then have from Lemma \ref{lemma 1}:
\begin{align*}
\psi(h(X))\psi(l(X))&=X^{-n}-b\\
X^k[\psi(h(X))\psi(l(X))]X^{n-k}&=1-X^k b X^{n-k}\\
h^*(X)\psi(l(X))X^{n-k}&=1- X^n\s^{k-n}(b)\\
&=[\s^{k-n}(b^{-1})-X^n]\s^{k-n}(b)\\
h^*(X)\psi(l(X))X^{n-k}\s^{k-n}(b^{-1})&=\s^{k-n}(b^{-1})-X^n\\
h^*(X)[-\psi(l(X))X^{n-k}\s^{k-n}(b^{-1})]&=X^n-\s^{k-n}(b^{-1}).
\end{align*}
Since $\mbox{deg}(l(X))=n-k$, $-\psi(l(X))X^{n-k}\s^{k-n}(b^{-1})\in A_\s$. It is now obvious that $h^*(X)$ is a left divisor of $X^n-\s^{k-n}(b^{-1})$. Now, keeping in mind that $\mbox{deg}(h^*(X))=\mbox{deg}(h(X))=k$ and the leading coefficient of $h^*(X)$ is $h_0\in U(A)$, it follows from Lemma \ref{lemma 2} that $h^*(X)$ is a right divisor of $X^n-a$, where $$a=\s^{n-k}(\s^{k-n}(b^{-1}))\s^{-k}(h_0)\s^{n-k}(h_0^{-1})=b^{-1}\s^{-k}(h_0)\s^{n-k}(h_0^{-1})$$ as claimed.
\end{proof}

\noindent{\bf Example 4.} Keep the notations of Example 3. By Lemma \ref{h*}, $h^*(X)= \al^3 X^3+\al^2 X^2+\al X+ 1$ is a left divisor of $X^4-\s^{-1}(\al^{-4})=X^4-\al^{-4}$. In fact, we have $$(\al^3 X^3+\al^2 X^2+\al X+ 1)(\al^{-3}X+\al^{-4})=X^4-\al^{-4}.$$
We also deduce from Lemma \ref{h*} that $h^*(X)$ is a right divisor of $X^4-\s^{3}(\al^3)/\al^4\s(\al^3)=X^4-\al^{-4}$ too. In fact, we
have $(\al^{-3}X+\al^{-4})(\al^3X^3+\al^2 X^2+\al X+ 1)=X^4-\al^{-4}$.$\\$

The following is a very important and interesting fact concerning the $A$-module orthogonal to a free $A$-module over a finite commutative ring $A$, where orthogonality is with respect to the Euclidean inner product.
\begin{lem}\label{orthogonal}
Let $A$ be a finite commutative ring with identity, $M$ an $A$-submodule of $A^n$, and $M^\perp$ the $A$-submodule of $A^n$ orthogonal to $M$ with respect to the Euclidean inner product on $A^n$. If $M$ is $A$-free of rank $k$, then $M^\perp$ is $A$-free of rank $n-k$.
\end{lem}
\begin{proof} See \cite[Proposition 2.9]{FLL}.
\end{proof}

In the terminology of this paper, \cite[Theorem 1]{BU2} characterizes the principal $\s$-codes over a finite field $\F$ (with $\s$ an automorphism of $\F$) whose duals are also principal $\s$-codes, extending \cite[Theorem 2]{BU1}. It is claimed in \cite[p. 240]{BU2} that \cite[Theorem 1]{BU2} remains valid over finite rings (not even assuming commutativity!) if one assumes that the constant term of $g$ is a unit. Yet, when looking at the proof of \cite[Theorem 1]{BU2}, we see that a crucial underlying assumption is that the dual of a linear code over a finite field is free (as both are vector spaces) and the sum of the dimensions of the two codes is equal to their length. However, the freeness assumption on the dual does not necessarily hold over rings in general even if the original linear code is free, let alone talking about the sum of the dimensions. So the same proof of \cite[Theorem 1]{BU2} can not be adopted for finite rings and, thus, we can not see at the moment how the aforementioned claim can be verified. To the best of the authors' knowledge, however, it was not until the appearance of \cite[Proposition 2.9]{FLL} (Lemma \ref{orthogonal} above) three years after \cite{BU2} that we were able to extend \cite[Theorem 1]{BU2} to {\it finite commutative} rings (Theorem \ref{dual 1} below).

\begin{thm}\label{dual 1}
Let $A$ be a finite commutative ring with identity, $\s$ a ring automorphism of $A$, and $\mathcal{C}$ a principal $\s$-code of length $n$ generated by some monic $g(X)=\sum_{i=0}^{n-k} g_i X^i\in A_\s$ with $g_0\in U(A)$.
\begin{itemize}
\item[(i)] If the dual $\mathcal{C}^\perp$ of $\mathcal{C}$ is a principal $\s$-code generated by some $h(X)=\sum_{i=0}^k h_i X^i\in A_\s$ with $h_0, h_k\in U(A)$, then $\mathcal{C}$ is principal $\s$-constacyclic with $\mathcal{C}=(g(X))_{n,\s}^{\s^k(g_0)\s^{2k}(h_k)}$.
\item[(ii)] If for some $a\in U(A)$, $\mathcal{C}=(g(X))_{n,\s}^a$ is principal $\s$-constacyclic, then the dual $\mathcal{C}^\perp$ of $\mathcal{C}$ is the principal $\s$-constacyclic code $\mathcal{C}^\perp=(h^*(X))_{n,\s}^{c}$, where $h(X)=\sum_{i=0}^k h_i X^i\in A_\s$ is such that $X^n-\s^{-k}(a)=g(X)h(X)$ with $h_0\in U(A)$, and $c=\s^{-k}(a^{-1})\s^{-k}(h_0)\s^{n-k}(h_0^{-1})$.
\end{itemize}
\end{thm}

\begin{proof}$\\$
\begin{itemize}
\item[(i)] Let $\mathcal{C}^\perp$ be a principal $\s$-code generated by some $h(X)=\sum_{i=0}^k h_i X^i \in A_\s$ with $\displaystyle{h_k, h_0\in U(A)}$. Since $h_0^{-1}h(X)\in A_\s$ also generates $\mathcal{C}^\perp$, we may assume that $h_0=1$. We let \linebreak $h^{\bot}(X)=\sum_{i=0}^k \s^{k-i}(h_{k-i})X^i$, and note that $h^\bot(X)$ is monic. We claim that $g(X)h^\bot(X)=X^n-g_0\s^k(h_k)$. Suppose that $g(X)h^\bot(X)=\sum_{i=0}^n c_i X^i$. Notice that $c_n=1$ and \linebreak $c_0 =g_0\s^k(h_k)$. To settle the claim, it remains to show that $c_l=0$ for $l\in \{1, \cdots, n-1\}$. Since $\{X^i g(X)\}_{0\leq i \leq k-1}$ and $\{X^jh(X)\}_{0 \leq j\leq n-k-1}$ are $A$-generators of $\mathcal{C}$ and $\mathcal{C}^\perp$, respectively, it follows that $$< X^{i_0} g(X)\,, X^{i_1} h(X)>\,=\,0$$ for any $i_0\in \{0, \cdots, k-1\}$ and $i_1\in \{0, \cdots, n-k-1\}$. So, for every such $i_0$ and $i_1$, we have
\begin{align*}
0 &=< X^{i_0} g(X)\,, X^{i_1} h(X)>\\
&=< \sum_{i=0}^{n-k} \s^{i_0}(g_i) x^{i+i_0}\, , \sum_{i=0}^k \s^{i_1}(h_i) X^{i+i_1}>\\
&=< \sum_{i=0}^{n-k} \s^{i_0}(g_i) x^{i+i_0}\, , \sum_{i=i_1-i_0}^{k+i_1 - i_0} \s^{i_1}(h_{i-i_1 +i_0}) X^{i+i_0} >\\
&= \sum_{i=\SMALL{\mbox{max}}\{0, i_1 -i_0\}}^{\SMALL{\mbox{min}}\{n-k, k+i_1-i_0\}} \s^{i_0}(g_i)\s^{i_1}(h_{i-i_1+i_0})\\
&=\s^{i_0}[\,\sum_{i=\SMALL{\mbox{max}}\{0, i_1 -i_0\}}^{\SMALL{\mbox{min}}\{n-k, k+i_1-i_0\}} g_i \s^{i_1-i_0}(h_{i-i_1+i_0})\,].
\end{align*}
Since $\s^{i_0}$ is an automorphism of $A$, $\sum_{i=\SMALL{\mbox{max}}\{0, i_1 -i_0\}}^{\SMALL{\mbox{min}}\{n-k, k+i_1-i_0\}} g_i \s^{i_1-i_0}(h_{i-i_1+i_0})=0$. Let \linebreak $l=k+i_1-i_0$. Then $l\in \{1, \cdots, n-1\}$ and $$\s^i(h_{l-i})=\s^i(\s^{l-i-k}(h_{k-l+i}))=\s^{l-k}(h_{k-l+i})=\s^{i-i_0}(h_{i-i_1+i_0}).$$ So,
\begin{align*}
0&=\sum_{i=\SMALL{\mbox{max}}\{0, i_1 -i_0\}}^{\SMALL{\mbox{min}}\{n-k, k+i_1-i_0\}} g_i \s^{i_1-i_0}(h_{i-i_1+i_0})\\&=\sum_{i=\SMALL{\mbox{max}}\{0,l-k\}}^{\SMALL{\mbox{min}}\{n-k, l\}} g_i \s^{l-k}(h_{k-l+i})\\ &=\sum_{i=\SMALL{\mbox{max}}\{0,l-k\}}^{\SMALL{\mbox{min}}\{n-k, l\}} g_i \s^i(h_{l-i})\\ &=c_l
\end{align*}
as desired. It now follows from Lemma \ref{lemma 2} that $X^n-\s^k(g_0)\s^{2k}(h_k)=\s^n(h^\bot(X))g(X)$ and, hence, $\mathcal{C}=(g(X))_{n,\s}^{\s^k(g_0)\s^{2k}(h_k)}$ is $\s$-constacyclic.
\item[(ii)] As $g(X)$ is a right divisor of $X^n-a$, it follows from Lemma \ref{lemma 2} that there exists some \linebreak $h(X)=\sum_{i=0}^k h_i X^i\in A_\s$ such that $X^n-\s^{-k}(a)=g(X)h(X)$. Since $g_0h_0=\s^{-k}(a)$ and $A$ is commutative with $\s^{-k}(a)\in U(A)$, $h_0\in U(A)$. It then follows from Lemma \ref{h*} that $h^*(X)$ is a right divisor of $X^n-c$ with $c=\s^{-k}(a^{-1})\s^{-k}(h_0)\s^{n-k}(h_0^{-1})$. Let $\mathcal{C}^*=(h^*(X))_{n,\s}^c$ be the principal $\s$-constacyclic code generated by $h^*(X)$. We show that $\mathcal{C}^*=\mathcal{C}^\perp$. As $\mathcal{C}$ is a principal $\s$-code generated by $g(X)$, which is of degree $n-k$, $\mathcal{C}$ is $A$-free of rank $k$ (\cite[Theorem 1]{BL}). Since $A$ is a finite commutative ring, it follows from Lemma \ref{orthogonal} that $\mathcal{C}^\perp$ is $A$-free of rank $n-k$. On the other hand, as $\mathcal{C}^*$ is a principal $\s$-code generated by $h^*(X)$, which is of degree $k$, $\mathcal{C}^*$ is $A$-free of rank $n-k$ too. So, $|\mathcal{C}^*|=|\mathcal{C}^\perp|<\infty$. It, thus, suffices to show that $\mathcal{C}^* \subseteq \mathcal{C}^\perp$. Since $\{X^i g(X)\}_{0\leq i \leq k-1}$ and $\{X^jh^*(X)\}_{0 \leq j\leq n-k-1}$ are $A$-generators of $\mathcal{C}$ and $\mathcal{C}^*$, respectively, it suffices to show that $<X^ig(X), X^jh^*(X)>\,=\,0$ for each such $i$ and $j$. An argument like that in part (i) above will do. Hence, $\mathcal{C}^\perp=(h^*(X))_{n,\s}^c$.
\end{itemize}
\end{proof}

\noindent{\bf Remark:} If we do not want to be so detailed on Theorem \ref{dual 1}, we would rephrase it as follows (with some obvious additions):
\begin{quote}
\it{Let $A$ be a finite commutative ring with identity, $\s$ a ring automorphism of $A$, and $\mathcal{C}$ a principal $\s$-code of length $n$ generated by some monic $g(X)=\sum_{i=0}^{n-k} g_i X^i\in A_\s$ with $g_0\in U(A)$. Then the following are equivalent (assuming in each case that the constant term of the generating skew-polynomial is a unit in $A$):

(i) $\mathcal{C}^\perp$ is a principal $\s$-code.

(ii) $\mathcal{C}^\perp$ is a principal $\s$-constacyclic code.

(iii) $\mathcal{C}$ is a principal $\s$-constacyclic code.}
\end{quote}
Note that "$(i) \to (iii)$" is part $(i)$ of Theorem \ref{dual 1}, "$(iii)\to (ii)$" is part $(ii)$ of Theorem \ref{dual 1}, and "$(ii)\to (i)$" is trivial.\\

\noindent{\bf Example 5.} Keep the notations of Examples 3 and 4. As $(X^3 +\al X^2 + \al^2 X + \al^3)(X-\al)=X^4 -\al^4$, let $\mathcal{C}$ be the principal $\s$-constacyclic code $\mathcal{C}=(X-\al)_{4, \s}^{\al^4}$. It then follows from Theorem \ref{dual 1} that $\mathcal{C}^\perp=(\al^3X^3 +\al^2X^2 +\al X +1)_{4, \s}^{\al^{-4}}$.\\

\noindent{\bf Remark:}
Note that in part (ii) of Theorem \ref{dual 1}, if $a\s^{-k}(a)=\s^{-k}(h_0)\s^{n-k}(h_0^{-1})$, then $\mathcal{C}^\perp$ is the principal $\s$-constacyclic code $\mathcal{C}^\perp=(h^*(X))_{n,\s}^{a}$. That is, both $\mathcal{C}$ and $\mathcal{C}^\perp$ are generated by right divisors of the same polynomial $X^n-a$.\\

If a $\s$-code $\mathcal{C}$ is principal $\s$-constacyclic over a finite commutative ring with identity (where $\s$ is an automorphism of the ring), Theorem \ref{dual 1} asserts that the dual code $\mathcal{C}^\perp$ is principal $\s$-constacyclic as well. The following theorem gives a generating matrix of the dual code in such a case.

\begin{cor}\label{main 2} Let $A$ be a finite commutative ring with identity, $\s$ a ring automorphism of $A$, $\displaystyle{a\in U(A)}$, and $\mathcal{C}=(g(X))^{a}_{n,\s}$ a principal $\s$-constacyclic code generated by some monic \linebreak $g(X)=\sum_{i=0}^{n-k} g_i X^i \in A_\s$ with $g_0 \in U(A)$. Let $h(X)=\sum_{i=0}^k h_i X^i \in A_\s$ be such that $g(X)h(X)= X^n-\s^{-k}(a)$, as ensured by Theorem \ref{dual 1}. Then a generating matrix $H\in M_{n-k, n}(A)$ of $\mathcal{C}^\perp$ is
$$\left( \begin{array}{ccccccc}
 h_{k} & \sigma(h_{k-1}) & \dots  &\sigma^{k}( h_0) & 0& \dots & 0\\
 0 & \sigma(h_{k}) &\sigma^2(h_{k-1})& \dots & \sigma^{k+1}(h_0) &0 & \dots  \\
  \vdots &  &  & \vdots & \vdots & & \vdots \\
  0 & \dots & \dots & 0 & \sigma^{n-k-1}(h_{k}) & \dots &\sigma^{n-1}(h_0) \end{array} \right).$$
\end{cor}
\begin{proof}
By Theorem \ref{dual 1}, the dual code $\mathcal{C}^\perp$ is a principal $\s$-constacyclic code generated by $h^*(X)$. Now, applying Corollary \ref{A_sigma} yields the desired conclusion.
\end{proof}

\noindent{\bf Example 6.}

\noindent (a) Keep the notations of Example 4. It follows from Corollary \ref{main 2} that a generating matrix of $\mathcal{C}^\perp$ is
$H=\left(\begin{array}{cccc} h_3 & \s(h_2) & \s^2(h_1) & \s^3(h_0)\end{array}\right)=\left(\begin{array}{cccc} 1 & \al & \al^2 & \al^3 \end{array}\right)$.

\noindent (b) Let $A=\left\{\left(\begin{array}{cc} a & b \\ 0 & a \\\end{array}\right)\,|\, a,b \in \Z_6 \right\}$ and $\s: \left(\begin{array}{cc}a & b \\ 0 & a \\\end{array}\right) \mapsto \left(\begin{array}{cc}a & -b \\ 0 & a \\\end{array}\right)$. Let
$\al =\left(\begin{array}{cc}1 & 1 \\
 0 & 1\\
\end{array}\right)\in U(A) $, $g(X)= X^2+\al \in A_\s$, $h(X)=X^2-\al \in A_\s$. We then get $$h(X)g(X)=g(X)h(X)=X^4-\al^2.$$ Letting $\mathcal{C}=(g(X))_{4,\s}^{\al^2}$, it follows from Theorem \ref{dual 1} that $\mathcal{C}^\perp=(h^*(X))_{4,\s}^{\al^{-2}}$ and from Corollary \ref{main 2} that $\mathcal{C}^\perp$ has the following generating matrix:
 $$H=\left(
      \begin{array}{cccc}
       h_2 & \s(h_1) & \s^2(h_0) & 0\\
        0 &\s(h_2) &\s^2(h_1)& \s^3(h_0)  \\
      \end{array}
    \right)=\left(
      \begin{array}{cccc}
       1 & \sigma(0) & \sigma^2( -\alpha) & 0\\
        0 &\sigma( 1) &\sigma^2( 0)& \sigma^3(-\alpha)  \\
      \end{array}
    \right)=\left(
      \begin{array}{cccc}
       1 & 0 &  -\al & 0\\
        0 &1 &0& -\al  \\
      \end{array}
    \right).$$\\

Due to Theorem \ref{dual 1}, the following result gives a characterization of self-dual $\s$-codes over finite commutative rings in such a way that generalizes \cite[Corollary 4]{BU1} and further strengthens it.

\begin{cor}\label{SDC} Keep the assumptions of Theorem \ref{dual 1} with $n=2k$. Then the following statements are equivalent:
\begin{itemize}
\item[(i)] $\mathcal{C}$ is a self dual $\s$-code.
\item[(ii)] $\mathcal{C}$ is a principal $\s$-constacyclic code with $\mathcal{C}=(g(X))_{n,\s}^a$, $a\in U(A)$, and $\displaystyle{\s^k(h_0^{-1})h^*(X)=g(X)}$, where $g(X)h(X)=X^n-\s^{-k}(a)$.
\item[(iii)] For any $l\in \{0,\cdots,k\}, \sum_{i=0}^l\s^{k-l}(g_i)g_{i+k-l}=0$.
\end{itemize}
\end{cor}

\begin{proof}$\\$
\indent \underline{$(\mbox{i}) \Leftrightarrow (\mbox{ii})$:} Assume that $\mathcal{C}=\mathcal{C}^\perp$. It follows from Theorem \ref{dual 1} and its proof that $\mathcal{C}^\perp$ is $\s$-constacyclic generated by $h^*(X)\in A_\s$, where $h(X)=\sum_{i=0}^k h_i X^i$ is satisfying $h_0\in U(A)$ and $g(X)h(X)=X^n-\s^{-k}(a)$ for some $a\in U(A)$. As $\s^k(h_0^{-1})h^*(X)$ also generates $\mathcal{C}^\perp$ and both $g(X)$ and $\s^k(h_0^{-1}) h^*(X)$ are monic and generate the same code, we must have $g(X)=\s^k(h_0^{-1})h^*(X)$. Conversely, assume that $\mathcal{C}=(g(X))_{n,\s}^a$, for some $a\in U(A)$, and $\s^k(h_0^{-1})h^*(X)=g(X)$ where $g(X)h(X)=X^n-\s^{-k}(a)$. Then, by Theorem \ref{dual 1}, $\mathcal{C}^\perp$ is principal and generated by $h^*(X)$. Since $h^*(X)$ and $\s^k(h_0^{-1})h^*(X)=g(X)$ generate the same code, we conclude that $\mathcal{C}=\mathcal{C}^\perp$.

\indent \underline{$(\mbox{i}) \Leftrightarrow (\mbox{iii})$:} Follow the proof of Corollary 4 of \cite{BU1} verbatim with the use of Theorem \ref{dual 1} and the obvious adjustments.
\end{proof}

\noindent{\bf Example 7.}
Let $A=\F_3\times \F_3$, $\s(x,y)=(y,x)$, and denote $(a,a)\in A$ by $a$. Taking $h(X)=X^2+2 X+2 \in A_\s$, we get $h^*(X)= 2 X^2+2 X+ 1$ and $\s^2(h_0^{-1})h^*(X)= 2(2 X^2+2 X+ 1)=X^2+ X+  2$. Letting $g(X)=X^2+ X+ 2$, a simple verification shows that  $g(X)h(X)= X^4+1$. We then deduce from Corollary \ref{SDC} (ii) that $\mathcal{C}=(X^2+ X+ 2)_{4,\s}^{-1}$ is a self-dual $\s$-constacyclic code over $A$, which is negacyclic over $A$ of length $4$. Using Magma (\cite{M}), this yields, after the obvious Gray map, a negacyclic $[8,4,3]$ ternary code over with generating matrix in systematic form
$\begin{pmatrix}
1 &0 &0 &0 &1 &0 &2 &0\\
0 &1& 0& 0 &0 &1 &0 &2\\
0 &0 &1 &0 &2 &0 &2 &0\\
0 &0& 0& 1& 0 &2 &0 &2
\end{pmatrix}$.
So, we get  a new construction of the unique self-dual code with these parameters \cite{G}, which is classically obtained as a direct sum of two copies of the tetracode \cite[Table XII]{R}.\\

\noindent{\bf Example 8.}\\
Consider $\F_4=\F_2(w)$ with $w^2+w+1=0$. Let $A=\F_4 \times \F_4$, $\sigma(x,y)=(x,y^2)$, and $f(X)=X^6+1\in A_\s$ where $1$ denotes $(1,1)$. Letting
\begin{align*}
g_1(X)&=h_1(X)=X^3+1\\
g_2(X)&=1+ (0, w^2)x+(0,w^2)x^2+ x^3\\
h_2(X)&=1+ (0, w^2)x+(0,w)x^2+ x^3\\
g_3(X)&=1+ (0, w)x+(0,w)x^2+ x^3\\
h_3(X)&=1+ (0, w)x+(0,w^2)x^2+ x^3,
\end{align*}
we see that $f(X)=g_i(X)h_i(x)$ for every $i=1,2,3$. So, by Corollary \ref{SDC}, the three   codes $C_1=(g_1(X))_{6,\s}^{-1}$, $C_2=(g_2(X))_{6,\s}^{-1}$, and $C_3=(g_3(X))_{6,\s}^{-1}$ are self-dual $\s$-constacyclic codes over $A$. Generating matrices of $C_1$, $C_2$, and $C_3$ are, respectively, as follows:
 $$G_1=\begin{pmatrix}
(1 ,1) &(0,0) &(0,0) &(1 ,1) &(0,0) &(0,0) \\
(0,0)&(1 ,1)& (0,0)& (0,0) &(1 ,1) &(0,0) \\
(0,0)&(0,0) &(1 ,1) &(0,0) &(0,0)&(1 ,1)\\
\end{pmatrix},$$
$$ G_2=\begin{pmatrix}
(1 ,1) &(0, w^2) &(0, w^2) &(1 ,1) &(0 ,0)&(0,0) \\
(0,0) &(1 ,1)&(0, w) &(0, w) &(1 ,1) &(0,0) \\
(0,0) &(0,0) &(1 ,1)&(0, w^2) &(0, w^2) &(1 ,1)\\
\end{pmatrix},$$
$$ G_3=\begin{pmatrix}
(1 ,1) &(0, w) &(0, w) &(1 ,1) &(0,0) &(0,0) \\
(0,0)&(1 ,1)&(0, w^2) &(0, w^2) &(1 ,1) &(0,0) \\
(0,0) &(0,0) &(1 ,1)&(0, w) &(0, w) &(1 ,1) \\
\end{pmatrix}.$$
Moreover, using the obvious Gray map to $\F_4$, we get from $C_2$ a self-dual $[12,6,2]$ code over $\F_4$. For this, Magma \cite{M} was used.

\section{{\bf{Control Matrix of a Principal $(f,\sigma,\delta)$-Code over a Ring}}}\label{control matrix}

We assume in this section that $A$ is a ring with identity, $\s$ is a ring endomorphism of $A$ that maps the identity to itself, and $\d$ is a $\s$-derivation of $A$.

For $H\in M_{n,t}(A)$ with $t\leq n$, denote by $\mbox{Ann}_l(H)$ the left $A$-submodule of $A^n$: $$\mbox{Ann}_l(H):=\{x\in A^n\,|\, xH=0\}.$$

If $\mathcal{C}$ is an $(f,\s,\d)$-code of length $n$ over $A$, a matrix $H \in M_{n,t}(A)$, with $t\leq n$, is called a {\it control matrix} of $\mathcal{C}$ if $\mathcal{C}=\mbox{Ann}_l(H)$. Consequently, for an $A$-free code $\mathcal{C}$, if $G$ is a generating matrix of $\mathcal{C}$ and $H$ is a control matrix of $\mathcal{C}$, then $GH=0$, from which it follows that the columns of $H$ are elements of the dual $\mathcal{C}^\perp$.

For a principal $(f,\s,\d)$-code $\mathcal{C}$ over $A$ that is generated by some monic $g(X)\in A_{\s, \d}$ which is both a right and left divisor of $f(X)$, Boulagouaz and Leroy in \cite{BL} gave a way of computing a control matrix of $\mathcal{C}$, as in Lemma \ref{control} below, using $T_f$ and $h(X)$, where $h(X)\in A_{\s, \d}$ is such that $g(X)h(X)=f(X)$. Theorem \ref{main 1} gives precise and more practical recursive formulas that compute a control matrix of $\mathcal{C}$ using $f(X)$, $h(X)$, $\s$, and $\d$. Corollary \ref{control delta zero} deals with the special case when $\d=0$, while Corollary \ref{control consta} handles the more special case when $\mathcal{C}$ is principal $\s$-constacyclic.

\begin{lem} \label{control} \cite[Corollary 1]{BL}
Let $\mathcal{C}$ be a principal $(f,\s, \d)$-code of length $n$ generated by some monic $g(X)\in A_{\s, \d}$ of degree $n-k$ which is also a left divisor of $f(X)$, with $f(X)=g(X)h(X)$ for some $h(X)=\sum_{i=0}^{k} h_i X^i\in A_{\s, \d}$. Then a control matrix of $\mathcal{C}$ is the matrix $H\in M_{n,n}(A)$ whose rows are $T_f^i(h_0, \cdots, h_k, 0, \dots, 0)$ for $0\leq i \leq n-1$.
\end{lem}$\\$
\noindent{\bf Remark:}  Lemma \ref{control} is still valid  if we assume that the leading coefficient of $g(X)$ is a unit in $A$.\\

With the assumptions of Lemma \ref{control}, the following theorem gives explicit and more practical recursive formulas to compute a control matrix.

\begin{thm} \label{main 1}
Keep the assumptions of Lemma \ref{control} with $f(X)=\sum_{i=0}^n a_i X^i$. Then, a control matrix $H\in M_{n,n}(A)$ of $\mathcal{C}$ is given by
$$\left( \begin{array}{cccccccc} h_0  & \dots  & h_{k} & 0& 0 &\dots & 0\\ h_0^{(1)} & \dots & h_{k}^{(1)} & \sigma(h_{k}) & 0 &  \dots & 0\\
h_0^{(2)} & \dots & h_{k}^{(2)} & h_{k+1}^{(2)} & \sigma^2(h_{k}) & \dots & 0\\
\vdots & \vdots & \vdots & \vdots & \vdots & \vdots & \vdots \\ h_0^{(n-k-1)} & \dots & h_{k}^{(n-k-1)} & h_{k+1}^{(n-k-1)} & h_{k+2}^{(n-k-1)} & \dots & \sigma^{n-k-1}(h_{k})\\  h_0^{(n-k)} & \dots & h_{k}^{(n-k)} & h_{k+1}^{(n-k)} & h_{k+2}^{(n-k)} & \dots & h_{n-1}^{(n-k)}\\ \vdots & \vdots & \vdots & \vdots & \vdots & \vdots & \vdots \\  h_0^{(n-1)} & \dots & h_{k}^{(n-1)} & h_{k+1}^{(n-1)} & h_{k+2}^{(n-1)} & \dots & h_{n-1}^{(n-1)} \end{array} \right),$$ where
\begin{itemize}
\item[(1)] $h_i=0$ for $k+1 \leq i \leq n-1$,
\item[(2)] for $1\leq i \leq n-k-1$ and $1 \leq j\leq n-2$,\\
\hspace*{1cm}(i) $h_0^{(i)} =\delta(h_0^{(i-1)})$,\\
\hspace*{1cm}(ii) $h_j^{(i)}=\delta(h_j^{(i-1)})+\sigma( h_{j-1}^{(i-1)})$,
\item[(3)] for $n-k \leq i \leq n-1$ and $1 \leq j\leq n-1$\\
\hspace*{1cm}(i) $h_0^{(i)} =\delta(h_0^{(i-1)})-a_0\sigma(h_{n-1}^{(i-1)})$, and\\
\hspace*{1cm}(ii) $h_j^{(i)}=\delta(h_j^{(i-1)})+\sigma( h_{j-1}^{(i-1)})-a_{j}\sigma(h_{n-1}^{(i-1)})$.
\end{itemize}
\end{thm}

\begin{proof}
By Lemma \ref{control}, a control matrix of $\mathcal{C}$ is the matrix $H\in M_{n,n}(A)$ whose rows are $$T_f^i(h_0, \cdots, h_k, 0, \dots, 0)=(h_0^{(i)}, \cdots, h_{n-1}^{(i)})$$ for $0\leq i \leq n-1$. Now applying Lemma \ref{lemma L} and Corollary \ref{enhance} with $s=k$ and $(h_0, \cdots, h_{n-1})$ in place of $(x_1, \cdots, x_n)$ yields the desired conclusion.
\end{proof}

\noindent{\bf Example 9.} Let $R$ be a ring of characteristic 3 with identity and $A=\left\{\left(\begin{array}{cc} a & b \\ 0 & a \\\end{array}\right)\,|\, a,b \in R \right\}$. Take $\s: \left(\begin{array}{cc}a & b \\ 0 & a \\\end{array}\right) \mapsto \left(\begin{array}{cc}a & b \\ 0 & a \\\end{array}\right)$ and $\d:\left(\begin{array}{cc}a & b \\ 0 & a \\\end{array}\right) \mapsto \left(\begin{array}{cc}0 & b \\ 0 & 0\\\end{array}\right)$. Let $f(X) =X^3+2X\in A_{\s, \d}$ where $2$ obviously denotes $2\left(\begin{array}{cc}1 & 0 \\ 0 & 1 \\\end{array}\right)$. Consider $g(X)=X + 2\beta\in A_{\sigma, \delta}$ and $h(X)=X^2+\beta X+\al \in A_{\sigma, \delta}$ with $\al=\left(\begin{array}{cc}0 & 1 \\ 0 & 0 \\\end{array}\right)$ and $\beta=\left(\begin{array}{cc}1& 1\\ 0 & 1\\\end{array}\right) $. A simple verification shows that $f(X) =g(X)h(X)=h(X)g(X)$. Let $\mathcal{C}$ be the principal $(f,\s,\d)$-code generated by $g(X)$. Noting that $h_0=\al$, $h_1=\beta$ and $h_2=1$, it follows from Theorem \ref{main 1} that a control matrix of $\mathcal{C}$ is
$$H=\left(
      \begin{array}{ccc}
        h_0 & h_1 & h_2 \\
         h_0^{(1)} & h_1^{(1)} & h_2^{(1)}  \\
          h_0^{(2)} & h_1^{(2)} & h_2^{(2)}  \\
      \end{array}
    \right)=\left(
      \begin{array}{ccc}
        \alpha & \beta & 1 \\
        \alpha  & 2\alpha+1& \beta \\
         \alpha  & \beta & 1 \\
      \end{array}
    \right).$$
To double check that $H$ is a correct matrix, it follows from Theorem \ref{main 0} that a generating matrix of $\mathcal{C}$ is
$G=\left(
      \begin{array}{ccc}
      2\beta & 1 & 0\\
      2\al & 2\beta & 1
      \end{array}
    \right).$
Now it can be easily checked that $GH=0$.


\begin{cor}\label{control delta zero}
Keep the assumptions of Theorem \ref{main 1} with $\d=0$. Then, a control matrix \linebreak $H\in M_{n,n}(A)$ of $\mathcal{C}$ is given by
$$\left( \begin{array}{ccccccccccc} h_0  & h_1 & h_2  &\dots  & h_{k} & 0& 0 & 0&\dots & 0\\ \textcolor[rgb]{0.00,0.07,1.00}{0} & \s(h_0) & \s(h_2) &\dots & \s(h_{k-1}) & \s(h_{k}) & 0 & 0 & \dots & 0\\
\textcolor[rgb]{0.00,0.07,1.00}{0} & \textcolor[rgb]{0.00,0.07,1.00}{0} & \s^2(h_0) & \dots & \s^2(h_{k-2}) & \s^2(h_{k-1}) & \s^2(h_{k}) & 0 & \dots & 0\\
\dots & \dots & \dots & \dots & \dots & \dots & \dots &\dots &\dots &\dots\\ \textcolor[rgb]{0.00,0.07,1.00}{0} &\textcolor[rgb]{0.00,0.07,1.00}{0} & \dots & \dots & \textcolor[rgb]{0.00,0.07,1.00}{0} & \s^{n-k-1}(h_0) &\dots &\dots & \dots &\sigma^{n-k-1}(h_{k})\\  h_0^{(n-k)} & h_1^{(n-k)} & \dots & \dots & h_{k}^{(n-k)} & h_{k+1}^{(n-k)} & h_{k+2}^{(n-k)} & \dots & \dots & h_{n-1}^{(n-k)}\\ \dots & \dots & \dots & \dots & \dots & \dots & \dots & \dots & \dots & \dots \\  h_0^{(n-1)} & h_1^{(n-1)} & \dots & \dots & h_{k}^{(n-1)} & h_{k+1}^{(n-1)} & h_{k+2}^{(n-1)} & \dots &\dots & h_{n-1}^{(n-1)} \end{array} \right),$$ where the number of initial consecutive zeros in the $i$th row is precisely $i-1$ for $i=2, \dots, n-k$, and
\begin{itemize}
\item[(1)] $h_i=0$ for $k+1 \leq i \leq n-1$,
\item[(2)] for $1\leq i \leq n-k-1$ and $1 \leq j\leq n-2$,\\
\hspace*{1cm}(i) $h_0^{(i)} =0$,\\
\hspace*{1cm}(ii) $h_j^{(i)}=\sigma( h_{j-1}^{(i-1)})$,

\item[(3)] for $n-k \leq i \leq n-1$ and $1 \leq j\leq n-1$,\\
\hspace*{1cm}(i) $h_0^{(i)} =-a_0\sigma(h_{n-1}^{(i-1)})$, and\\
\hspace*{1cm}(ii) $h_j^{(i)}=\sigma( h_{j-1}^{(i-1)})-a_{j}\sigma(h_{n-1}^{(i-1)})$.
\end{itemize}
\end{cor}

\begin{proof}
Use Theorem \ref{main 1} and Corollary \ref{enhance 2}.
\end{proof}

\begin{cor}\label{control consta}
Keep the assumptions of Corollary \ref{control delta zero}. Let $\mathcal{C}$ be a principal $\s$-constacyclic code \linebreak $\mathcal{C}=(g(X))_{n, \s}^a$ for some $a\in U(A)$ such that $g(X)$ is also a left divisor of $X^n-a$ with \linebreak $X^n-a=g(X)h(X)$ for some $h(X)=\sum_{i=0}^{k} h_i X^i\in A_{\s, \d}$. Then, the entries of a control matrix \linebreak $H=(H_{i,j})\in M_{n,n}(A)$ of $\mathcal{C}$ are as follows:
\begin{itemize}
\item[(a)] If $n-k=1$, then
$$H_{1,j}=  \begin{array} {c@{\quad;\quad}l} h_{j-1} & \mbox{if}\;\; 1\leq j \leq n
\end{array} $$ and, for $2\leq i \leq n$,
$$H_{i,j}= \left \{ \begin{array} {c@{\quad;\quad}l} -\s^{j-1}(a_0)\s^{i-1}(h_{n-i+j}) & \mbox{if}\;\; 1\leq j \leq i-1 \\
\s^{i-1}(h_{j-i}) & \mbox{if}\;\; i \leq j\leq n.
\end{array} \right. $$

\item[(b)] If $n-k \geq 2$, then
\begin{itemize}
\item[(i)] for $i=1$,
$$H_{1,j}= \left \{ \begin{array} {c@{\quad;\quad}l} h_{j-1} & \mbox{if}\;\; 1\leq j \leq k+1 \\
0 & \mbox{if}\;\; k+2\leq j\leq n.
\end{array} \right. $$

\item[(ii)] for $2 \leq i \leq n-k$,
$$H_{i,j}= \left \{ \begin{array} {c@{\quad;\quad}l} 0 & \mbox{if}\;\; 1 \leq j \leq i-1 \\
\s^{i-1}(h_{j-i}) & \mbox{if}\;\; i \leq j\leq i+k \\
0 & \mbox{if}\;\; i+k+1 \leq j \leq n.
\end{array} \right. $$

\item[(iii)] for $n-k+1 \leq i \leq n$,
$$H_{i,j}= \left \{ \begin{array} {c@{\quad;\quad}l} -\s^{j-1}(a_0)\s^{i-1}(h_{n-i+j}) & \mbox{if}\;\; 1\leq j \leq i-(n-k) \\
0 & \mbox{if}\;\; i-(n-k)+1 \leq j\leq i-1 \\
\s^{i-1}(h_{j-i}) & \mbox{if}\;\; i \leq j \leq n.
\end{array} \right. $$
\end{itemize}
\end{itemize}
\end{cor}

\begin{proof}
Apply Corollary \ref{enhance 3} and Corollary \ref{control delta zero}.
\end{proof}

\section{{\bf Parity-Check Matrix of a Principal $(f,\s,\d)$-Code over a Finite Commutative Ring}}\label{parity check matrix}
Let $A$ be a ring, $\s$ a ring endomorphism of $A$ that maps the identity to itself, and $\d$ a $\s$-derivation of $A$. If $\mathcal{C}$ is an $A$-free $(f, \s, \d)$-code of length $n$ and rank $k$, a matrix $H_*\in M_{n-k, n}(A)$ is called a {\it parity-check matrix} of $\mathcal{C}$ if:

\indent\indent (1) $H_*^T$ is a control matrix of $\mathcal{C}$, and

\indent\indent (2) $H_*$ is a generating matrix of the dual $\mathcal{C}^\perp$.

In classical coding theory over finite fields, the dual code of a linear code is also linear and, hence, a parity-check matrix of such a code always exists. However, for a principal $(f, \s, \d)$-code $\mathcal{C}$ over a ring $A$ (despite being $A$-free), the dual $\mathcal{C}^\perp$ may not be $A$-free and, thus, a parity-check matrix of $\mathcal{C}$ may not exist (due to the lack of requirement (2) above). Nonetheless, when $A$ is a finite commutative ring with identity and $\s$ is a ring automorphism of $A$, nice things happen. With this assumption added to the hypotheses of Theorem \ref{main 1}, Theorem \ref{main 3} below shows that the transpose of the matrix consisting of the last $n-k$ columns of $H$ of Theorem \ref{main 1} is indeed a parity-check matrix of $\mathcal{C}$. This is a dramatic improvement of Theorem $\ref{main 1}$ in this important and widely used case.

\begin{thm}\label{main 3}
Let $A$ be a finite commutative ring with identity, $\s$ a ring automorphism of $A$, and keep the other notations and assumptions of Theorem \ref{main 1}. Then, a parity-check matrix \linebreak $H_*\in M_{n-k, n}(A)$ of $\mathcal{C}$ is given by
$$\left( \begin{array}{cccccccc} h_{k} & h_{k}^{(1)} & h_{k}^{(2)} & \dots & h_{k}^{(n-k-1)} & h_{k}^{(n-k)} & \dots & h_{k}^{(n-1)}\\ 0  & \s(h_{k}) & h_{k+1}^{(2)} & \dots & h_{k+1}^{(n-k-1)} & h_{k+1}^{(n-k)} & \dots & h_{k+1}^{(n-1)}\\ 0  & 0 & \s^2(h_{k}) & \dots & h_{k+2}^{(n-k-1)} & h_{k+2}^{(n-k)} & \dots & h_{k+2}^{(n-1)} \\ \vdots & \vdots & \vdots & \vdots & \vdots & \vdots & \vdots & \vdots \\ 0  & 0 & 0 & \dots & \s^{n-k-1}(h_{k}) & h_{n-1}^{(n-k)} & \dots & h_{n-1}^{(n-1)} \end{array} \right),$$
where the $h_j^{(i)}$ are as in Theorem \ref{main 1}.
\end{thm}

\begin{proof}
Note that $H_*$ is the transpose of the last $n-k$ columns of $H$ of Theorem \ref{main 1}. The rows $H_1, \dots, H_{n-k}$ of $H_*$ are $A$-linearly independent since $H_*$ is in echelon form. Let $\mathcal{C}_*$ be the free left $A$-submodule of $A^n$ a basis of which is $H_1, \dots, H_{n-k}$. Then $\mathcal{C}^*$ has cardinality equal to $|A|^{n-k}$. On the other hand, it follows from Lemma \ref{orthogonal} that $\mathcal{C}^{\perp}$ is $A$-free of rank $n-k$. So $\mathcal{C}^\perp$ has cardinality equal to $|A|^{n-k}$ as well. With $H$ of Theorem \ref{main 1}, we have $\mathcal{C}=\mbox{Ann}_l(H) \subseteq \mbox{Ann}_l(H_*^T)$. By Lemma \ref{orthogonal} again, $\mbox{Ann}_l(H_*^T)$ is $A$-free of rank $k$. Then $|\mbox{Ann}_l(H_*^T)|=|A|^{k}=|\mathcal{C}|$ and so $\mathcal{C}=\mbox{Ann}_l(H_*^T)$. Thus, $H_*^T$ is a control matrix of $\mathcal{C}$. This, in particular, implies that $H_1, \dots, H_{n-k}\in \mathcal{C}^\perp$. So $\mathcal{C}^*\subseteq \mathcal{C}^\perp$. Since $\mathcal{C}^*$ and $\mathcal{C}^\perp$ are of the same finite cardinality, $\mathcal{C}^*=\mathcal{C}^\perp$. Thus, $H_*$ is a generating matrix of $\mathcal{C}^\perp$. Hence, $H_*$ is a parity-check matrix of $\mathcal{C}$.
\end{proof}

\noindent{\bf Example 10.} Keep the notation and assumptions of Example 9 with $A$ finite and commutative and $\s$ a ring automorphism of $A$. By Theorem \ref{main 3}, the matrix $H_*=\left(
      \begin{array}{ccc}
        1 & \beta & 1
      \end{array}
    \right)$
is a parity-check matrix of $\mathcal{C}$. By Theorem \ref{main 0}, a generating matrix of $\mathcal{C}$ is
$G=\left(
      \begin{array}{ccc}
      2\beta & 1 & 0\\
      2\al & 2\beta & 1
      \end{array}
    \right).$
It can be easily checked that $GH_*^T=0$.\\


When $\d=0$, $H_*$ takes a nicer form.

\begin{cor}\label{parity sigma code}
Keep the assumptions of Theorem \ref{main 3} with $\d=0$. Then, a parity-check matrix $H_*\in M_{n-k, n}(A)$ of the principal $\s$-code $\mathcal{C}$ is given by
$$\left( \begin{array}{ccccccccc} h_{k} & \s(h_{k-1}) & \s^2(h_{k-2})& \dots & \s^k(h_0) & h_k^{(k+1)} & h_k^{(k+2)} & \dots & h_{k}^{(n-1)}\\ 0  & \s(h_{k}) & \s^2(h_{k-1}) & \dots & \s^k(h_1) & \s^{k+1}(h_0) & h_{k+1}^{(k+2)} & \dots & h_{k+1}^{(n-1)}\\ \dots & \dots  & \dots & \dots & \dots & \dots & \dots & \dots & \dots \\ 0  & 0 & 0 & \dots & \dots & \s^{n-k-1}(h_k) & \dots & \dots & \s^{n-1}(h_0) \end{array} \right),$$
where the $h_j^{(i)}$ are as in Theorem \ref{main 1}.
\end{cor}

\begin{proof}
Follows immediately from Theorem \ref{main 3} and Corollary \ref{control delta zero}.
\end{proof}

A special, yet important, case of Corollary \ref{parity sigma code} is when $\mathcal{C}$ is a principal $\s$-constacyclic code, in which case $H_*$ takes a much better form.

\begin{cor}\label{parity-consta}
Keep the assumptions of Theorem \ref{main 3} with $\mathcal{C}$ a principal $\s$-constacyclic code, $\mathcal{C}=(g(X))_{n,\s}^a$ for some $a\in U(A)$. Then, a parity-check matrix $H_*\in M_{n-k, n}(A)$ of $\mathcal{C}$ is given by
$$\left( \begin{array}{ccccccccc} h_{k} & \s(h_{k-1}) & \s^2(h_{k-2})& \dots & \s^k(h_0) & 0 & 0 & \dots & 0\\ 0  & \s(h_{k}) & \s^2(h_{k-1}) & \dots & \s^k(h_1) & \s^{k+1}(h_0) & 0 & \dots & 0\\ \dots & \dots  & \dots & \dots & \dots & \dots & \dots & \dots & \dots \\ 0  & 0 & 0 & \dots & \dots & \s^{n-k-1}(h_k) & \dots & \dots & \s^{n-1}(h_0) \end{array} \right).$$
\end{cor}

\begin{proof}
For $j=1, \dots, n$, let $C_j$ denote the $j$th column of $H$ in Corollary \ref{control consta}. If $n-k=1$, then $H_*=(C_{k+1}^T)$, where (by Corollary \ref{control consta} (a)) $$C_{k+1}^T= C_n^T= (h_k, \s(h_{k-1}), \s^2(h_{k-2}), \dots, \s^k(h_0)).$$ Suppose now that $n-k \geq 2$. Then the rows of $H_*$ are precisely $C_{k+1}^T, C_{k+2}^T, \dots, C_n^T$. We begin by specifying the entries of $C_{k+1}$, where we show that
$$C_{k+1}^T=(H_{i,k+1})_{_{1\leq i \leq n}}^T=(h_k, \s(h_{k-1}), \s^2(h_{k-2}), \dots, \s^k(h_0), 0 , \dots, 0),$$
that is, $H_{i, k+1}= \s^{i-1}(h_{k-(i-1)})$ for $1 \leq i \leq k+1$, and $H_{i,k+1}=0$ for $k+2 \leq i \leq n$. By Corollary \ref{control consta} (b)-(i), $H_{1, k+1}=h_k$. Let $2\leq i \leq n$. We deal with the following three cases:
\begin{itemize}
\item[] \underline{Case ($j=k+1=n-k$):} For $2\leq i \leq n-k=k+1$, we have $2\leq i \leq k+1=j \leq i+k$, so
we are in the second case of Corollary \ref{control consta} (b)-(ii). Thus, $H_{i, k+1}=\s^{i-1}(h_{k-(i-1)})$ here. For $n-k+1 \leq i \leq n$, we have $j=k+1=n-k \leq i-1 \leq n-1$ and $i-(n-k)+1 \leq k+1= j$. So $i-(n-k)+1 \leq j \leq i-1$, and we are in the second case of Corollary \ref{control consta} (b)-(iii). Thus, $H_{i,k+1}=0$ here. This fully verifies the asserted entries of $C_{k+1}$ when $j=k+1=n-k$.

\item[] \underline{Case ($j=k+1>n-k$):} For $2 \leq i \leq n-k$, we have $i\leq n-k < j=k+1$ and $i+k \geq 2+k
>j$. So $i\leq j \leq i+k$, and we are in the second case of Corollary \ref{control consta} (b)-(ii). Thus, $H_{i, k+1}=\s^{i-1}(h_{j-i})=\s^{i-1}(h_{k-(i-1)})$ here. Let $n-k+1 \leq i \leq n$. If $i\leq j \leq n$, then we are in third case of Corollary \ref{control consta} (b)-(iii). Thus, $H_{i,k+1}=\s^{i-1}(h_{j-i})=\s^{i-1}(h_{k-(i-1)})$ here as well. If $n-k+1 \leq j \leq i-1$, then $i-(n-k)+1 \leq n-(n-k)+1=k+1=j$. So $i-(n-k)+1 \leq j \leq i-1$, and we are in the second case of Corollary \ref{control consta} (b)-(iii). Thus, $H_{i,k+1}=0$ here. This fully verifies the asserted entries of $C_{k+1}$ when $j=k+1 > n-k$.

\item[] \underline{Case ($j=k+1 < n-k$):} Let $2 \leq i \leq n-k$. If $i\leq k+1 =j <n-k$, then $i+k \leq
n-k+k=n$. So we have $i\leq j \leq i+k$, and we are in the second case of Corollary \ref{control consta} (b)-(ii). Thus, $H_{i,k+1}=\s^{i-1}(h_{k+1-i})=\s^{i-1}(h_{k-(i-1)})$ here. If $j=k+1 <i \leq n-k$, then $1\leq j \leq i-1$, and we are in the first case of Corollary \ref{control consta} (b)-(ii). Thus, $H_{i,k+1}=0$ here. For $n-k+1 \leq i \leq n$, we have $j=k+1< n-k < n-k+1 \leq i$. So, $j \leq i-1$. Also, $i-(n-k)+1 \leq n-(n-k)+1=k+1=j$. So we have $i-(n-k)+1 \leq j \leq i-1$, and we are in the second case of Corollary \ref{control consta} (b)-(iii). Thus, $H_{i,k+1}=0$ here as well. This fully verifies the asserted entries of $C_{k+1}$ when $j=k+1 <n-k$.
\end{itemize}
Now, as for $C_{k+1+t}$ with $t=1, \dots, n-k-1$, note that (by Corollary \ref{control consta}) $H_{i, k+1+t}=0$ for $1\leq i \leq t$. For $t+1\leq i \leq k+1+t$, Corollary \ref{enhance 3} (a) (new version) yields $H_{i, k+1+t}=\s^{i-1}(h_{k+1+t-i})$. For $k+1+t+1 \leq i \leq n$, Corollary \ref{enhance 3} (a) again yields $H_{i, k+1+t}=0$. This completes the proof.
\end{proof}

Note that a requirement in the above corollary is that $g(X)$ be both a right and left divisor of $X^n-a$ (according to Theorem \ref{main 3}). The following corollary deals with the case when $g(X)$ is a right divisor of $X^n-a$ and a left divisor of $X^n-\s^{-k}(a)$ and $g_0\in U(A)$ (see the assumptions of Corollary \ref{main 2}).

\begin{cor}\label{parity-consta2}
Keep the assumptions of Corollary \ref{main 2}. Then, a parity-check matrix $H_*\in M_{n-k, n}(A)$ of $\mathcal{C}$ is given by
$$\left( \begin{array}{ccccccccc} h_{k} & \s(h_{k-1}) & \s^2(h_{k-2})& \dots & \s^k(h_0) & 0 & 0 & \dots & 0\\ 0  & \s(h_{k}) & \s^2(h_{k-1}) & \dots & \s^k(h_1) & \s^{k+1}(h_0) & 0 & \dots & 0\\ \dots & \dots  & \dots & \dots & \dots & \dots & \dots & \dots & \dots \\ 0  & 0 & 0 & \dots & \dots & \s^{n-k-1}(h_k) & \dots & \dots & \s^{n-1}(h_0) \end{array} \right).$$
\end{cor}

\begin{proof}
By Corollary \ref{main 2}, $H_*$ is a generating matrix of $\mathcal{C}^\perp$. Furthoermore, it is clear that $z\in\mbox{Ann}_l(H_*^T)$ if and only if $z\in \{x\in A^n\,|\, <x,y>=0 \;\mbox{for all}\; y\in \mathcal{C}^\perp\}=(\mathcal{C}^\perp)^\perp$. Since $(\mathcal{C}^\perp)^\perp$ and $\mathcal{C}$ are both free of the same rank (thanks to Lemma \ref{orthogonal}) and $\mathcal{C}\subseteq (\mathcal{C}^\perp)^\perp$, we conclude that $\mathcal{C}=\mbox{Ann}_l(H_*^T)$. Hence, $H_*$ is a parity-check matrix of $\mathcal{C}$ as claimed.
\end{proof}

\section*{Acknowledgement}
The authors would like to thank Patrick Sol\'{e} and Felix Ulmer for some useful conversations and suggestions.

\end{document}